\documentclass[a4paper,reqno,11pt]{amsart}

\usepackage{fullpage}

\usepackage{amssymb}
\usepackage{dsfont}

\usepackage[utf8]{inputenc} 
\usepackage[T1]{fontenc} 

\usepackage{hyperref}

\usepackage[nobysame,alphabetic,initials,msc-links]{amsrefs}

\DefineSimpleKey{bib}{how}

\renewcommand{\eprint}[1]{#1}
\BibSpec{misc}{%
  +{}{\PrintAuthors}  {author}
  +{,}{ \textit}      {title}
  +{,}{ }             {how}
  +{}{ \parenthesize} {date}
  +{,} { available at \eprint}        {eprint}
  +{,}{ available at \url}{url}
  +{,}{ }             {note}
  +{.}{}              {transition}
}

\numberwithin{equation}{section}

\theoremstyle{plain}
\newtheorem{thm}{Theorem}[section]
\newtheorem{prop}[thm]{Proposition}
\newtheorem{lemma}[thm]{Lemma}
\newtheorem{cor}[thm]{Corollary}
\theoremstyle{definition}

\newtheorem{defn}[thm]{Definition}
\theoremstyle{remark}

\newtheorem{question}[thm]{Question}
\newtheorem{remark}[thm]{Remark}
\newtheorem{example}[thm]{Example}

\theoremstyle{plain}
\newtheorem{innercustomthm}{Theorem}
\newenvironment{customthm}[1]
  {\renewcommand\theinnercustomthm{#1}\innercustomthm}
  {\endinnercustomthm}

\newcommand\bp{\begin{proof}}
\newcommand\ep{\end{proof}}

\newcommand{\un}{\mathds{1}}

\newcommand\C{\mathbb{C}}

\newcommand{\G}{{\mathcal{G}}}
\newcommand{\HH}{{\mathcal{H}}}

\newcommand{\V}{{\mathcal{V}}}

\newcommand\Ad{\operatorname{Ad}}
\newcommand{\CcG}{{C_{c}(\mathcal{G})}}

\newcommand{\Gu}{{\mathcal{G}^{(0)}}}
\newcommand{\Gxx}{\mathcal{G}^x_x}

\newcommand\ev{\mathrm{ev}}
\newcommand\Ind{\operatorname{Ind}}

\newcommand\supp{\operatorname{supp}}
\newcommand\Prim{\operatorname{Prim}}
\newcommand\ess{\mathrm{ess}}
\newcommand\sing{\mathrm{sing}}
\newcommand\eps{\varepsilon}
\newcommand\Iso{\operatorname{Iso}(\mathcal G)^{\circ}}
\newcommand\IsoG{\operatorname{Iso}(\mathcal G)}
\newcommand\IsoGx[1]{\operatorname{Iso}(\mathcal G_{\overline{[#1]}})}

\newcommand\ee{\nopagebreak\mbox{\ }\hfill$\diamond$}

\begin{document}

\title{Isotropy fibers of ideals in groupoid C$^{*}$-algebras}

\date{August 21, 2023; revised March 22, 2024}

\author{Johannes Christensen}
\address{Department of Mathematics, KU Leuven, Belgium}
\email{johannes.christensen@kuleuven.be}

\author{Sergey Neshveyev}
\address{Department of Mathematics, University of Oslo, Norway}
\email{sergeyn@math.uio.no}

\thanks{J.C. is supported by the postdoctoral fellowship 1291823N of the Research Foundation Flanders. S.N. is partially supported by the NFR funded project 300837 ``Quantum Symmetry''.}

\begin{abstract}
Given a locally compact \'etale groupoid and an ideal $I$ in its groupoid C$^*$-algebra, we show that $I$ defines a family of ideals in group C$^*$-algebras of the isotropy groups and then study to which extent $I$ is determined by this family. As an application we obtain the following results: (a) prove that every proper ideal is contained in an induced primitive ideal; (b) describe the maximal ideals; (c) classify the primitive ideals for a class of graded groupoids with essentially central isotropy.
\end{abstract}

\maketitle

\section*{Introduction}

In recent work we associated possibly exotic group C$^{*}$-algebras $C_{e}^{*}(\Gxx)$, $x\in \Gu$, to the isotropy groups $\Gxx$ of a not necessarily Hausdorff locally compact \'etale groupoid $\mathcal{G}$~\cite{CN}. Our motivation was analysis of KMS states and weights on groupoid C$^*$-algebras~\citelist{\cite{N},\cite{MR4592883}}. In this paper we use this construction to study the ideal structure of such algebras.

Our starting point is the observation that the image $I_x$ of any ideal $I\subset C^*_r(\G)$ under the canonical contraction $C^*_r(\G)\to C^*_e(\Gxx)$ is an ideal, which we call an isotropy fiber of~$I$. The observation is not entirely new: for transformation groupoids  $\G=X\rtimes\Gamma$ (when $C_{e}^{*}(\Gxx)=C_{r}^{*}(\Gamma_x)$ by~\cite{CN}), the fact that $I_x$ is at least an algebraic ideal in~$C^*_r(\Gamma_x)$ has been used in~\citelist{\cite{O},\cite{Ka}}. A natural question then is whether the ideals~$I_x$ contain complete information about $I$. A similar question can be asked for any groupoid C$^*$-algebra $C^*_\nu(\G)$ defined by a C$^*$-norm $\|\cdot\|_\nu$ dominating the reduced norm, once we define appropriate group C$^*$-algebras $C^*_\nu(\Gxx)$, see Section~\ref{ssec:completions} for the precise definition of $C^*_\nu(\Gxx)$.

This turns out to be related to the old question whether every primitive ideal is induced from an isotropy group. It was conjectured by Effros and Hahn~\cite{EH} that this is the case for the second countable transformation groupoids $X\rtimes\Gamma$ with $\Gamma$ amenable. The conjecture was verified, in a generalized form, by Sauvageot~\cite{Sau} and Gootman--Rosenberg~\cite{GR}. Extensions of their results to more general groupoids were then given by Renault~\cite{R} and Ionescu--Williams~\cite{IW}. As a consequence our question has a positive answer for all amenable second countable Hausdorff \'etale groupoids.

In the nonamenable case the question has in general a negative answer, see Examples~\ref{ex:AFS0}, \ref{ex:AFS} below. Nevertheless by looking at the family $(I_x)_x$ we can obtain the following result.

\begin{customthm}{A}\label{thm:A}
Assume $\G$ is a locally compact \'{e}tale groupoid, $C^*_\nu(\G)$ is a groupoid C$^*$-algebra and $I\subset C^*_\nu(\G)$ is a proper ideal. Let $U\subset\Gu$ be the open set such that $C_0(\Gu)\cap I=C_0(U)$. Then, for every point $x\in U^c$, there is a primitive ideal $J\subset C^*_\nu(\Gxx)$ such that $I\subset\Ind J$.
\end{customthm}

Even for amenable groupoids Theorem~\ref{thm:A} provides information about the ideal structure that is not immediately available from the Effros--Hahn conjecture. Our two main applications are as follows.

First we show that it is possible, at least in principle, to classify the maximal ideals in terms of the dynamics on the unit space and the isotropy groups.
In other words, we can describe the closed points of the primitive spectrum.

\begin{customthm}{B}\label{thm:B}
The maximal ideals of $C^*_\nu(\G)$ are parameterized by the pairs $(F,J)$, where $F\subset \Gu$ is a minimal closed invariant subset and $J\subset C^*_\nu(\G^{x_F}_{x_F})$ is a maximal $\G$-invariant ideal for a fixed point $x_F\in F$.
\end{customthm}

We refer the reader to Section~\ref{sectionMaximal} for further discussion and an explanation of the terminology used here.

Our second application is a classification of all primitive ideals for a class of groupoids.

\begin{customthm}{C}\label{thm:C}
Assume $\G$ is an amenable second countable Hausdorff locally compact \'etale groupoid graded by a discrete abelian group $\Gamma$, with grading $\Phi\colon\G\to \Gamma$ such that $\Phi\colon\Gxx\to \Gamma$ is injective for all~$x$. Then the primitive ideals of $C^*(\G)$ are parameterized by $(\Gu\times\hat \Gamma)/_\sim$, where $(x_1,\chi_1)\sim (x_2,\chi_2)$ if and only if $\overline{[x_1]}=\overline{[x_2]}$ and $\chi_1=\chi_2$ on $\Phi(\IsoGx{x_1}^\circ_{x_1})=\Phi(\IsoGx{x_2}^\circ_{x_2})$.
\end{customthm}

This generalizes the description of $\Prim (C_0(X)\rtimes \Gamma)$ due to Williams~\cite{MR0617538} (but without describing the topology) and the result of Sims and Williams~\cite{SW} on the Deaconu--Renault groupoids $\G_T$ defined by tuples $T=(T_1,\dots,T_k)$ of commuting local homeomorphisms, which in turn subsumes earlier results on $k$-graph C$^*$-algebras.
In fact, we prove a more general result than Theorem~\ref{thm:C} that allows for a small degree of noncommutativity in the isotropy groups, see Section~\ref{sec:graded} for the precise statement.

In addition to our two main applications, we have included two sections where we further investigate some consequences of our results. In Section~\ref{sec:simplicity} we analyze what can be said about an ideal $I\subset C^*_r(\G)$ if $C^*_r(\HH)\cap I=0$ for an open subgroupoid $\HH\subset\G$. Our results here complement and generalize a number of results in the literature. In Section~\ref{sec:exoticfiber} we show that under a mild countability assumption the kernel of the canonical homomorphism $C_{e}(\G_{x}^{x})\to C_{r}^{*}(\G_{x}^{x})$ coincides with the isotropy fiber of the singular ideal in $C^*_r(\G)$ for a residual set of $x\in \Gu$.

\bigskip

\section{Preliminaries} \label{sec:prelim}

\subsection{\'Etale groupoids and their \texorpdfstring{C$^*$}{C*}-algebras}

Let $\G$ be a groupoid with unit space $\Gu$, the range map $r\colon g\mapsto gg^{-1}$ and the source map $s\colon g\mapsto  g^{-1}g$. For sets $D,F \subset \Gu$ we write $\G_{D}^{F}:=r^{-1}(F)\cap s^{-1}(D)$, and in particular we write $\G_{x}^{x}:=\G_{\{x\}}^{\{x\}}$ for the isotropy group at $x\in \Gu$. A subset $U\subset \Gu$ is called invariant if $U=s(r^{-1}(U))$. 

\smallskip

As in \cite{CN}, we will work with possibly non-Hausdorff locally compact \'{e}tale groupoids, i.e., we assume that the topology on $\G$ satisfies the following properties:
\begin{enumerate}
\item[-] the groupoid operations are continuous;
\item[-] the unit space $\Gu $ is a locally compact Hausdorff space in the relative topology;
\item[-] the map $r$ is a local homeomorphism.
\end{enumerate}
For several results we will need some additional assumptions on the topology on $\G$, such as second countability, but we will state them explicitly every time they are needed.

We denote by $C_c(\G)$ the space spanned by functions $f$ on $\G$ such that~$f$ is zero outside an open Hausdorff subset $W\subset \G$ and the restriction of $f$ to $W$ is continuous and compactly supported. When $\G$ is non-Hausdorff, this space contains functions that are not continuous, so to denote it by $C_c(\G)$ is a slight abuse of notation. One can find other notations for this function space in the literature, like $C_c(\G)_{0}$ or $\mathcal{C}(\G)$, but we will denote it by $C_c(\G)$ to keep our notation in line with our previous paper \cite{CN} and make the paper as accessible as possible to the readers only interested in Hausdorff groupoids.

The space $C_c(\G)$ is a $*$-algebra with convolution product
\begin{equation*} \label{eprod}
(f_{1}*f_{2})(g) := \sum_{h \in \G^{r(g)}} f_{1}(h) f_{2}(h^{-1}g)
\end{equation*}
and involution by $f^{*}(g):=\overline{f(g^{-1})}$.

For every point $x\in \Gu $ we have a representation $\rho_{x}\colon C_{c}(\G) \to B(\ell^{2}(\G_{x}))$ defined by
\begin{equation}\label{eq:rhox}
\rho_{x}(f) \delta_{g} :=\sum_{h \in \G_{r(g)}} f(h) \delta_{hg},
\end{equation}
where $\delta_g$ is the Dirac delta-function. Equivalently,
\begin{equation}\label{eq:rhox2}
(\rho_x(f)\xi)(g) =
\sum_{h\in \G^{r(g)}}f(h) \xi(h^{-1}g),\qquad \xi\in\ell^2(\G_x).
\end{equation}
The reduced norm $\lVert \cdot \rVert_{r}$ on $C_{c}(\G)$ is defined by
\begin{equation*}
\lVert f \rVert_{r}  := \sup_{x\in \Gu } \lVert \rho_{x}(f) \rVert,
\end{equation*}
and the reduced groupoid $C^{*}$-algebra $C^{*}_{r}(\G)$ is then defined as the completion of $C_{c}(\G)$ with respect to this norm.

\begin{defn}\label{def:groupoid-algebra}
By a \emph{groupoid C$^*$-algebra} we mean the completion $C^*_\nu(\G)$ of $C_c(\G)$ with respect to some C$^*$-norm $\|\cdot\|_\nu$ dominating the reduced norm.
\end{defn}

Recall that a bisection $W$ is a subset of $\G$ such that $r|_{W}\colon W\to r(W)$ and $s|_{W}\colon W\to s(W)$ are bijections. It is well-known that if $f\in C_c(W)$ for an open bisection $W\subset\G$, then
\begin{equation} \label{eq:norm-bisection}
\|f\|_\nu=\lVert f\rVert_{\infty}
\end{equation}
for any C$^*$-norm $\|\cdot\|_\nu$ as above, see, e.g., \cite{MR2419901}*{Section~3}. Here we of course view $f$ as a function on $\G$ by extending it by $0$ outside $W$. This identity implies in particular that there is a maximal C$^*$-norm $\|\cdot\|$, and we denote by $C^*(\G)$ the completion of~$C_c(\G)$ with respect to this norm.

\smallskip

Given a groupoid C$^*$-algebra $C^*_\nu(\G)$ we get a groupoid C$^*$-algebra $C^*_\nu(\G_X)$ for every invariant subset $X\subset\Gu$ that is either open or closed. Namely, assume first that $X$ is open and invariant. We can view~$C_c(\G_X)$ as a subspace of $C_c(\G)$ by extending functions by zero. Then $C_c(\G_X)$ becomes a $*$-ideal in $C_c(\G)$ and we get by restriction a C$^*$-norm on $C_c(\G_X)$, which we continue to denote by~$\|\cdot\|_\nu$. For every $x\in X$, the restriction of the regular representation $\rho_x\colon C_c(\G)\to B(\ell^2(\G_x))$ to $C_c(\G_X)$ coincides with the regular representation of $C_c(\G_X)$ associated with $x$. Hence the norm $\|\cdot\|_\nu$ on~$C_c(\G_X)$ dominates the reduced norm.

Assume now that $X$ is closed and invariant. Define a C$^*$-seminorm on $C_c(\G_X)$ by
$$
\|f\|_\nu:=\sup_\pi\|\pi(f)\|,\quad f\in C_c(\G_X),
$$
where the supremum is taken over all representations $\pi$ of $C_c(\G_X)$ such that the representation $f\mapsto \pi(f|_{\G_X})$ of $C_c(\G)$ extends to $C^*_\nu(\G)$. As for every $x\in X$ we have $(\G_X)_x=\G_x$, the representation~$\rho_x$ of~$C_c(\G)$ factors through $C_c(\G_X)$ and defines the regular representation of~$C_c(\G_X)$ associated with $x$. Hence the seminorm $\|\cdot\|_\nu$ on $C_c(\G_X)$ is actually a norm dominating the reduced norm.

It is not difficult to see that the two constructions agree when $X$ is both open and closed, as then $C_c(\G)=C_c(\G_X)\bigoplus C_c(\G_{X^c})$.

\begin{prop}\label{prop:exact}
Assume $\G$ is a locally compact \'{e}tale groupoid, $C^*_\nu(\G)$ is a groupoid C$^*$-algebra and $X\subset\Gu$ is a closed invariant set. Then we have a short exact sequence
$$
0\to C^*_\nu(\G_{X^c})\to C^*_\nu(\G)\to C^*_\nu(\G_X)\to0
$$
of groupoid C$^*$-algebras.
\end{prop}

\bp
By construction $C^*_\nu(\G_{X^c})$ can be viewed as an ideal in $C^*_\nu(\G)$. For $f\in \CcG$, we have
$$
\lVert f|_{\G_{X}} \rVert_{\nu} =\sup_{\pi} \lVert \pi(f|_{\G_{X}})\rVert \leq \lVert f \rVert_{\nu},
$$
where $\pi$ runs through the representations described above, so the map~$f\mapsto f|_{\G_X}$ extends to a surjective homomorphism $C^*_\nu(\G)\to C^*_\nu(\G_X)$. Clearly~$C^*_\nu(\G_{X^c})$ is contained in the kernel of this homomorphism. In order to show that it coincides with this kernel it suffices to check that the map $C_c(\G)\to C^*_\nu(\G)/C^*_\nu(\G_{X^c})$ factors through $C_c(\G_X)$. For this, in turn, it is enough to show that $C_c(\G_{X^c})$ is dense in the kernel of the map $C_c(\G)\to C_c(\G_X)$, $f\mapsto f|_{\G_X}$, with respect to the norm $\|\cdot\|_\nu$ on~$C_c(\G)$. By \cite{NS}*{Proposition~1.1} this is true for the maximal C$^*$-norm, hence for~$\|\cdot\|_\nu$.
\ep

It is well-known that if we start with the maximal C$^*$-norm on $C_c(\G)$, then the norms that we get on $C_c(\G_X)$ are again maximal, see, e.g., \cite{NS}*{Corollary~1.2}. On the other hand, if we take the reduced norm, then for open $X$ we obtain the reduced norm, but for closed $X$ we get in general a different norm on $C_c(\G_{X})$, even when $X$ is a single point~\citelist{\cite{HLS},\cite{W}} or $\G$ is principal~\cite{AFS} (see also Example~\ref{ex:AFS} below). One therefore has to exercise some caution with our notation: if one chooses $\lVert \cdot \rVert_{\nu}=\lVert \cdot \rVert_{r}$ on $\CcG$, the C$^{*}$-algebra $C^*_\nu(\G_X)$ we have introduced for closed invariant $X$ does not necessarily coincide with $C^*_{r}(\G_X)$. When we do get the reduced norm on $C_c(\G_X)$ for all closed invariant subsets $X\subset\Gu$, then $\G$ is called \emph{inner exact}. By~\cite{KW}*{Theorem~5.2} every transformation groupoid $X\rtimes\Gamma$, where $\Gamma$ is an exact discrete group, is inner exact.

\subsection{Induced representations}\label{ssec:induced}

Assume $\G$ is a locally compact \'etale groupoid. Fix a unit $x\in\Gu$. Then every unitary representation $\pi\colon\Gxx\to B(H)$ of the isotropy group $\Gxx$ can be induced to a representation of $C_c(\G)$. It is convenient to have two pictures of this construction.

\smallskip

In the first picture we consider the Hilbert space $\Ind H$ of functions $\xi \colon  \G_{x} \to H$ such that
\begin{equation*}
\xi(gh)=\pi(h)^{*}\xi(g),\quad g \in \G_{x},\ h\in \G_{x}^{x},
\end{equation*}
and
\begin{equation*}
\sum_{g\in \G_{x}/\G_{x}^{x}}\lVert \xi(g)\rVert^{2}<\infty .
\end{equation*}
The induced representation $\Ind\pi=\Ind^\G_{\Gxx}\pi\colon C_c(\G)\to B(\Ind H)$ is then given by the same formula as in~\eqref{eq:rhox2}. If $\lambda_{\Gxx}$ is the left regular representation of $\Gxx$ we have a unitary equivalence $\Ind\lambda_{\Gxx}\sim\rho_x$. To see this, let $\G_x\times_{\Gxx}\Gxx$ be the quotient of $\G_x\times\Gxx$ by the equivalence relation $(gh',h)\sim (g,h'h)$  ($g\in\G_x$, $h,h'\in\Gxx$). Then $\Ind \ell^{2}(\Gxx)$ can be identified  with $\ell^{2}(\G_x\times_{\Gxx}\Gxx)$, and the canonical bijection $\G_x\times_{\Gxx}\Gxx\to\G_x$, $(g,h)\mapsto gh$, gives rise to a unitary intertwiner between $\Ind\lambda_{\Gxx}$ and~$\rho_x$.

\smallskip

For the second picture consider the restriction map
$$
\eta_x\colon C_c(\G)\to\C\Gxx=C_c(\Gxx),\quad \eta_x(f):=f|_{\Gxx},
$$
and define a $\C\Gxx$-valued sesquilinear form on $C_{c}(\G)$ by
$$
\langle f_1 ,f_2 \rangle_{\Gxx}:=\eta_x( f^{*}_1*f_2),\quad f_1,f_2 \in C_{c}(\G).
$$
Note that $\eta_x(f^{*}_1*f_2)$ depends only on the restrictions of $f_1$ and $f_2$ to $\G_x$, so we actually get a form on $C_c(\G_x)$. Then define a pre-inner product on $C_{c}(\G_{x})\otimes H$ by
$$
(f_1 \otimes \xi_1 , f_2 \otimes \xi_2):
= \langle \xi_1 , \pi(\langle f_1, f_2 \rangle_{\Gxx})\xi_2 \rangle,
$$
where we extended $\pi$ to the group algebra $\C\Gxx$ by linearity. By factoring out the subspace of zero length vectors and completing we obtain a Hilbert space $H_{0}$.

\begin{lemma}
There exists a unitary operator $U\colon H_{0} \to \Ind H$ with $U[f\otimes \eta]=\xi_{f,\eta}$ for all $f\in C_{c}(\G_{x})$ and $\eta \in H$, where
$$
\xi_{f, \eta}(g):=\sum_{h\in\Gxx}f(gh)\pi(h)\eta,\quad g\in\G_x.
$$
Under the identification of $H_0$ with $\Ind H$ we thus get, the induced representation is given by
$$
(\Ind \pi)(q) [f\otimes \eta] = [\rho_x(q)f \otimes \eta ],\quad q\in C_c(\G),\ f\in C_c(\G_x),\ \eta\in H.
$$
\end{lemma}

\bp
The map $C_{c}(\G_{x})\times H \to \Ind H$, $(f,\eta) \mapsto \xi_{f, \eta}$, is bilinear, so by the universal property of the tensor product we get a linear map $ C_{c}(\G_{x})\otimes H \to \Ind H $ that maps $f\otimes \eta$ to $\xi_{f, \eta}$. Since
\begin{align*}
( f_{1}\otimes \eta_{1} , f_{2} \otimes \eta_{2} )
&=  \Big\langle \sum_{h\in \Gxx} (f_{2}^{*}*f_{1})(h) \pi(h) \eta_{1} ,  \eta_{2} \Big \rangle
= \Big\langle \sum_{h\in \Gxx} \sum_{g\in \G_{x} } \overline{f_{2}(g)} f_{1}(gh) \pi(h) \eta_{1} ,  \eta_{2} \Big \rangle \\
&= \sum_{g\in \G_{x} /\G_{x}^{x}} \Big\langle \sum_{h_1\in \Gxx} f_{1}(gh_1) \pi(h_1) \eta_{1} , \sum_{h_2\in \Gxx} f_{2}(gh_2) \pi(h_2) \eta_{2} \Big \rangle
=\langle\xi_{f_{1}, \eta_{1}} , \xi_{f_{2}, \eta_{2}} \rangle \; ,
\end{align*}
this linear map gives rise to an isometry $U\colon H_{0} \to \Ind H$. It is not difficult to see that $U$ has dense image, so $U$ is unitary. For the last statement of the lemma notice that $(\Ind \pi) (q) \xi_{f, \eta}=\xi_{\rho_{x}(q)f, \eta}$.
\ep

\subsection{Completions of the isotropy group algebras}\label{ssec:completions}

Assume that $\G$ is a locally compact \'etale groupoid, $C^*_\nu(\G)$ is a groupoid C$^*$-algebra and $x\in\Gu$. As we saw above, if the closed set~$\{x\}$ is invariant, then the norm on $C^*_\nu(\G)$ defines a norm on $\C\Gxx$, which we continue to denote by~$\|\cdot\|_\nu$. Using induced representations we can extend this to points $x\in \Gu$ that are not necessarily invariant.

Namely, for $a\in\C\Gxx$ define
$$
\|a\|_\nu:=\sup_\pi\|\pi(a)\|,
$$
where the supremum is taken over all unitary representations $\pi$ of $\Gxx$ such that the induced representation $\Ind\pi$ of $C_c(\G)$ extends to $C^*_\nu(\G)$. The collection of such representations $\pi$ contains the regular representation  $\lambda_{\Gxx}$ of $\Gxx$, since $\Ind\lambda_{\Gxx}\sim\rho_x$. Hence $\|\cdot\|_\nu$ is a C$^*$-norm on $\C\Gxx$ dominating the reduced norm. Denote by $C^*_\nu(\Gxx)$ the completion of $\C\Gxx$ with respect to $\|\cdot\|_\nu$.

If we start with the maximal norm on $C_c(\G)$, then obviously we get the maximal C$^*$-norm on~$\C\Gxx$. On the other hand, if we start with the reduced norm, then in general, as we have already remarked, we can get a norm on $\C\Gxx$ different from the reduced one. Following~\cite{CN} we denote this norm by~$\|\cdot\|_e$, with $e$ for \emph{exotic}.

\begin{lemma}\label{lem:contraction}
The restriction map $\eta_x\colon C_c(\G)\to\C\Gxx$ extends to a contractive completely positive map
$
\vartheta_{x,\nu}\colon C^*_\nu(\G)\to C^*_\nu(\Gxx).
$
\end{lemma}

\bp
This is proved identically to \cite{CN}*{Lemma~1.2}: since the collection of unitary representations~$\pi$ of $\Gxx$ such that $\Ind\pi$ extends to $C_{\nu}^{*}(\G)$ is closed under direct sums, we can take a faithful representation $\pi$ of $C^*_\nu(\Gxx)$ on $H$ such that $\Ind \pi$ extends to $C_{\nu}^{*}(\G)$. If we identify~$C^*_\nu(\Gxx)$ with $\pi(C^*_\nu(\Gxx))$, then the map $\vartheta_{x,\nu}$ is given by $a\mapsto v(\Ind\pi)(a)v^*$, where $v\colon\Ind H\to H$ is the coisometry given by $v\xi=\xi(x)$.
\ep

For future reference we will write out the following corollary, which follows directly from the proof of Lemma~\ref{lem:contraction}.

\begin{cor} \label{cor:contraction}
If $\pi$ is a representation of $C^*_\nu(\Gxx)$ on a Hilbert space $H$ and $v\colon\Ind H\to H$ is the coisometry given by $v\xi=\xi(x)$ for $\xi \in \Ind H$, then
$
\vartheta_{x,\nu}\colon C^*_\nu(\G)\to C^*_\nu(\Gxx)
$
has the property that
$
v(\Ind\pi)(a)v^* = \pi(\vartheta_{x,\nu}(a))
$
for all $a\in C_{\nu}^{*}(\G)$.
\end{cor}

As in \cite{CN}, we denote the maps $\vartheta_{x,\nu}$ for the maximal and reduced norms on $C_c(\G)$ by
$$
\vartheta_{x}\colon C^*(\G)\to C^*(\Gxx)\qquad\text{and}\qquad \vartheta_{x,e}\colon C^*_r(\G)\to C^*_e(\Gxx),
$$
respectively.

By~\cite{CN}*{Lemma~1.4}, we have the following useful property.

\begin{lemma} \label{lem:multdomain}
Assume that $g_{1}, g_{2}, \dots, g_{n}$ are distinct points in $\Gxx$ and $\{W_{i}\}_{i=1}^{n}$ is a family of open bisections of $\G$ with $g_{i}\in W_{i}$ for each $i$.
Assume $f \in C_{c}(\G)$ is zero outside $\bigcup_{i=1}^{n} W_{i}$. Then~$f$ lies in the multiplicative domain of $\vartheta_{x,\nu}$.
\end{lemma}

\begin{cor}\label{cor:surj}
If $\mathfrak M_{x,\nu}$ is the multiplicative domain of the map $\vartheta_{x,\nu}\colon C^*_\nu(\G)\to C^*_\nu(\Gxx)$, then $\vartheta_{x,\nu}(\mathfrak M_{x,\nu})=C^*_\nu(\Gxx)$. 
\end{cor}

\bp
By the previous lemma the C$^*$-algebra $\vartheta_{x,\nu}(\mathfrak M_{x,\nu})$ is dense in $C^*_\nu(\Gxx)$, hence it coincides with $C^*_\nu(\Gxx)$.
\ep

The following description of the norm on $\C\Gxx$ will be crucial for us.

\begin{thm}\label{thmnorm}
Assume $\G$ is a locally compact \'{e}tale groupoid, $C^*_\nu(\G)$ is a groupoid C$^*$-algebra and $x\in\Gu$. Then, for every $a\in \C\Gxx$, we have
\begin{equation}\label{eq:norm-inf}
\lVert a \rVert_{\nu} = \inf\left\{ \lVert f \rVert_{\nu} : f\in C_{c}(\G),\ \eta_{x}(f)=a    \right\}.
\end{equation}
Furthermore, let $\V$ be a neighbourhood base at $x$ partially ordered by containment. For each $V\in\V$, choose a function $q_V\in C_c(\Gu)$ such that $q_V(x)=1$, $0\le q_V\le 1$ and $\supp q_V\subset V$. For $a\in \C\Gxx$, choose any function $f\in C_c(\G)$ such that $\eta_x(f)=a$. Then
\begin{equation}\label{eq:e-norm-limit}
\|a\|_\nu=\lim_{V\in\V}\|q_V*f*q_V\|_\nu.
\end{equation}
\end{thm}

\bp
The proof is similar to that of~\cite{CN}*{Theorem 2.4} for the reduced norm on $C_c(\G)$, so we will be sketchy.

First one proves that for any $a\in\C\Gxx$ the limit in~\eqref{eq:e-norm-limit} exists and equals the right hand side of~\eqref{eq:norm-inf}, which defines a seminorm~$\|\cdot\|_{\nu'}$ on $\C\Gxx$. This is proved exactly as \cite{CN}*{Lemma~2.5} using property~\eqref{eq:norm-bisection} and Lemma~\ref{lem:multdomain}. As~$\vartheta_{x,\nu}$ is a contraction, $\|\cdot\|_{\nu'}$ is actually a norm dominating $\|\cdot\|_\nu$. Then one checks, using  Lemma~\ref{lem:multdomain}, that $\|\cdot\|_{\nu'}$ is a C$^*$-norm in the same way as in \cite{CN}*{Lemma~2.6}.

To prove that $\|\cdot\|_{\nu'}=\|\cdot\|_\nu$, we will show that if $\varphi$ is a state on $\C\Gxx$ bounded with respect to~$\|\cdot\|_{\nu'}$, then it is also bounded with respect to~$\|\cdot\|_\nu$. We start by observing that by the definition of $\|\cdot\|_{\nu'}$, if $\varphi$ is a state on $\C\Gxx$ bounded with respect to $\|\cdot\|_{\nu'}$, then $\varphi\circ\eta_x$ is a positive linear functional on $C_c(\G)$ bounded with respect to $\|\cdot\|_\nu$. By \cite{CN}*{Lemma~1.3} the corresponding GNS-representation of $C_c(\G)$, which is bounded with respect to $\|\cdot\|_\nu$, is the induction $\Ind\pi_\varphi$ of the GNS-representation~$\pi_\varphi$ (viewed as a representation of~$\Gxx$). Therefore~$\pi_\varphi$ lies in the collection of representations defining the norm $\|\cdot\|_\nu$ on $\C\Gxx$. Hence $\varphi$ is bounded with respect to this norm, and since this is true for all $\varphi$, we must have $\|\cdot\|_{\nu'}=\|\cdot\|_\nu$.
\ep

Similarly to \cite{CN}*{Corollaries~2.7 and~2.8} we then get the following result, which can be viewed as a substitute for Proposition~\ref{prop:exact}.

\begin{cor}\label{cor:exact}
The space $C_{c}(\G\setminus\Gxx)$ is dense in $\ker(\vartheta_{x,\nu}\colon C^*_\nu(\G)\to C^*_\nu(\Gxx))$.
\end{cor}

Note that the space $C_{c}(\G\setminus\Gxx)$ here is defined in the same way as $C_c(\G)$.

\begin{remark} \label{rem:groupbundle}
If $\G$ is a group bundle, meaning that $\G^y_x=\emptyset$ for all $y\ne x$, then the C$^*$-algebra~$C^*_\nu(\G)$ is a $C_0(\Gu)$-algebra. Corollary~\ref{cor:exact} or Proposition~\ref{prop:exact} imply then that $C^*_\nu(\Gxx)$ is the fiber of~$C^*_\nu(\G)$ at~$x$.\ee
\end{remark}

Next, assume $\HH\subset\G$ is an open subgroupoid. Denote by $C^*_\nu(\mathcal H)$ the closure of $C_c(\mathcal H)$ in~$C^*_\nu(\G)$. This is a groupoid C$^*$-algebra in the sense of Definition~\ref{def:groupoid-algebra}, since $C^*_r(\mathcal H)\subset C^*_r(\G)$ by \cite{MR2134336}*{Proposition~1.9}. We remark that the groupoids in~\cite{MR2134336} are assumed to be Hausdorff, but this is not used in the proof of this result. Observe then that identity~\eqref{eq:e-norm-limit} implies that for $x\in\HH^{(0)}$ the norm on~$\C\HH^x_x$ defined by $C^*_\nu(\HH)$ coincides with the restriction of the norm on $\C\Gxx$ defined by $C^*_\nu(\G)$.

\subsection{Factorization of states through isotropy groups}

We have the following characterization of states arising from isotropy groups.

\begin{prop} \label{prop:pointeval}
Assume $\G$ is a locally compact \'{e}tale groupoid, $C^*_\nu(\G)$ is a groupoid C$^*$-algebra and $x\in\Gu$.
Let $\psi$ be a state on $C_{\nu}^{*}(\mathcal{G})$. Then
$
\psi=\varphi \circ \vartheta_{x,\nu}
$
for a state $\varphi$ on $C_{\nu}^{*}(\Gxx)$ if and only if the restriction of $\psi$ to $C_{0}(\Gu)$ is the point-evaluation~$\ev_x$.
\end{prop}

\begin{proof}
If $\psi=\varphi \circ \vartheta_{x,\nu}$, then clearly $\psi|_{C_0(\Gu)}=\ev_x$. Conversely, assume $\psi|_{C_0(\Gu)}=\ev_x$. Take $f\in C_c(\G)$. As $C_0(\Gu)$ lies in the multiplicative domain of $\psi$, for every $q\in C_0(\Gu)$ with $q(x)=1$ we get
$\psi(q*f*q)=\psi(f)$, hence $|\psi(f)|\le\|q*f*q\|_\nu$. By~\eqref{eq:e-norm-limit} we conclude that
$$
|\psi(f)|\le\|\eta_x(f)\|_\nu.
$$
It follows that $\psi=\varphi \circ \vartheta_{x,\nu}$ for a norm one linear functional $\varphi$ on $C^*_\nu(\Gxx)$. As $\varphi(1)=1$, this functional must be positive.
\end{proof}

There are a number of results in the literature related to this factorization, see, for example, \citelist{\cite{MR0656488}*{Proposition~6.1}\cite{AS}*{Lemma~2}\cite{BKKO}*{Lemma~7.9}\cite{BNRSW}*{Theorem~3.1(a)}}. It is sometimes convenient to view it as a unique state extension property.

\begin{cor}
Assume $\G$ is a locally compact \'{e}tale groupoid, $C^*_\nu(\G)$ is a groupoid C$^*$-algebra  and $\mathcal H\subset\G$ is an open subgroupoid. Let $C^*_\nu(\mathcal H)$ be the closure of $C_c(\mathcal H)$ in $C^*_\nu(\G)$. Assume $\psi$ is a state on $C^*_\nu(\mathcal H)$ satisfying $\psi|_{C_0(\mathcal H^{(0)})}=\ev_x$ for a point $x\in \mathcal H^{(0)}$ such that $\mathcal H^x_x=\Gxx$. Then $\psi$ extends uniquely to a state on $C^*_\nu(\G)$.
\end{cor}

\bp
Since the point-evaluation at $x$ on $C_0(\mathcal H^{(0)})$ extends uniquely to a state on $C_0(\Gu)$, by Proposition~\ref{prop:pointeval} any state extension of $\psi$ must be of the form $\varphi\circ\vartheta_{x,\nu}$ for a state $\varphi$ on~$C^*_\nu(\Gxx)$. As $\eta_x(C_c(\mathcal H))=\C\Gxx$ is dense in $C^*_\nu(\Gxx)$, the state $\varphi$ is uniquely determined by $\psi$.
\ep

\bigskip

\section{Isotropy fibers of ideals}

Assume $\G$ is a locally compact \'{e}tale groupoid and $C^*_\nu(\G)$ is a groupoid C$^*$-algebra. Assume $I\subset C^*_\nu(\G)$ is an ideal, by which we will always mean a closed two-sided ideal. Define
$$
I_x:=\vartheta_{x,\nu}(I)\subset C^*_\nu(\Gxx),\quad x\in\Gu.
$$

\begin{lemma} 
For every $x\in\Gu$, $I_x$ is an ideal in $C_\nu^{*}(\Gxx)$.
\end{lemma}

\begin{proof}
We claim that the norm on $\vartheta_{x,\nu}(I)$ coincides with the quotient norm on $I/(I\cap\ker\vartheta_{x,\nu})$ under the identification of these spaces by $\vartheta_{x,\nu}$. As $\vartheta_{x,\nu}$ is a contraction, it is clear that the quotient norm dominates $\|\cdot\|_\nu$. On the other hand, if $a\in I$ and $(q_V)_V$ is as in Theorem~\ref{thmnorm}, then it follows from~\eqref{eq:e-norm-limit} that
$$
\|\vartheta_{x,\nu}(a)\|_\nu=\lim_V\|q_Vaq_V\|_\nu.
$$
As $q_Vaq_V\in I$ and $\vartheta_{x,\nu}(q_Vaq_V)=\vartheta_{x,\nu}(a)$, this shows that the norm $\|\cdot\|_\nu$ is at least as large as the quotient norm on $I/(I\cap\ker\vartheta_{x,\nu})$. Hence the two norms coincide, proving the claim. It follows that the normed space $\vartheta_{x,\nu}(I)$ is complete, hence it is closed in $C^*_\nu(\Gxx)$.

Next, let $a= \vartheta_{x , \nu}(a')$ with $a'\in I$ and take $b\in C_{\nu}^{*}(\Gxx)$. By Corollary~\ref{cor:surj} there exists an element $b'$ in the multiplicative domain of $\vartheta_{x, \nu}$ such that $\vartheta_{x, \nu}(b')=b$. We then have that $a'b', b'a'\in I$ and $ab = \vartheta_{x,\nu}(a' b')$, $ba = \vartheta_{x , \nu}(b'a')$, proving that $I_x$ is an ideal.
\end{proof}

\begin{defn}
We call the ideal $I_x\subset C^*_\nu(\Gxx)$ the \emph{isotropy fiber} of $I$ at $x$.
\end{defn}

The family of ideals $I_x$, $x\in\Gu$, is invariant under the action of $\G$ by conjugation in the following sense. Recall that for any $\gamma \in \G$ we have an isomorphism $\Ad\gamma\colon\C\G_{s(\gamma)}^{s(\gamma)}\to\C \G_{r(\gamma)}^{r(\gamma)}$.

\begin{lemma}\label{lem:invariance}
For every $\gamma\in\G$, the isomorphism $\Ad\gamma\colon\C\G_{s(\gamma)}^{s(\gamma)}\to\C \G_{r(\gamma)}^{r(\gamma)}$ extends to an isomorphism $\Psi_{\gamma,\nu} \colon C^*_\nu(\G_{s(\gamma)}^{s(\gamma)})\to C^*_\nu(\G_{r(\gamma)}^{r(\gamma)})$, and we have $\Psi_{\gamma,\nu}(I_{s(\gamma)})=I_{r(\gamma)}$.
\end{lemma}

\bp If $\pi$ is a unitary representation of $\G_{r(\gamma)}^{r(\gamma)}$ then $\pi\circ\Ad\gamma$ is a unitary representation of~$\G_{s(\gamma)}^{s(\gamma)}$ and the right translation by $\gamma$ defines an equivalence between $\Ind\pi$ and $\Ind(\pi\circ\Ad\gamma)$. It therefore follows from the definition of the norms on the isotropy group algebras that $\Ad\gamma$ extends to an isomorphism $\Psi_{\gamma,\nu} \colon C^*_\nu(\G_{s(\gamma)}^{s(\gamma)})\to C^*_\nu(\G_{r(\gamma)}^{r(\gamma)})$.

For the second statement, choose an open bisection $W$ containing $\gamma$. Let $f\in C_{c}(W)$ be such that $f(\gamma)=1$. Then a straightforward computation shows that
$$
\eta_{r(\gamma)}(f*q*f^{*}) = (\Ad\gamma)(\eta_{s(\gamma)}(q))\quad \text{for all}\quad q\in \CcG.
$$
Hence
\begin{equation*} 
\vartheta_{r(\gamma),\nu}(f\cdot f^{*}) =\Psi_{\gamma,\nu}\circ \vartheta_{s(\gamma),\nu}.
\end{equation*}
It follows that $\Psi_{\gamma,\nu}(I_{s(\gamma)})\subset I_{r(\gamma)}$. As $\Psi_{\gamma,\nu}^{-1}=\Psi_{\gamma^{-1},\nu}$, we get the equality.
\ep

Let us say that a family of ideals $J_x\subset C^*_\nu(\Gxx)$, $x\in\Gu$, is \emph{invariant} if it has the property from the lemma above: $\Psi_{\gamma,\nu}(J_{s(\gamma)})=J_{r(\gamma)}$ for all $\gamma$.

\smallskip

Therefore we have shown that for every ideal $I\subset C^*_\nu(\G)$ its isotropy fibers form an invariant family of ideals. It is natural to ask the following:

\begin{question} \label{ques1}
When are the ideals of $C^*_\nu(\G)$ determined by their isotropy fibers? In other words, when can we say that given ideals $I_1$ and $I_2$ such that $I_{1,x}=I_{2,x}$ for all $x\in\Gu$, we must have $I_1=I_2$?
\end{question}

The question can be formulated in several ways. Assume $\mathcal J=(J_x)_{x\in\Gu}$ is an invariant family of ideals $J_x\subset C^*_\nu(\Gxx)$. Define
\begin{equation}\label{eq:I(J)}
I(\mathcal J):=\{ a\in C_{\nu}^{*}(\G) : \vartheta_{x , \nu }(a^{*}a)\in J_{x} \text{ for all } x\in \Gu \}.
\end{equation}
As we will see shortly, $I(\mathcal J)$ is an ideal in $C^*_\nu(\G)$. Assuming that this is true, let us first establish the following result.

\begin{lemma}\label{lem:equiv}
For any locally compact \'{e}tale groupoid $\G$ and any groupoid C$^*$-algebra $C^*_\nu(\G)$, the following properties are equivalent:
\begin{enumerate}
  \item the ideals of $C^*_\nu(\G)$ are determined by their isotropy fibers;
  \item for every ideal $J\subset C^*_\nu(\G)$ with isotropy fibers $J_x$, we have $J=I(\{J_x\}_{x\in\Gu})$;
  \item every ideal of $C^*_\nu(\G)$ is of the form $I(\mathcal J)$ for some invariant family of ideals $\mathcal J$.
\end{enumerate}
\end{lemma}

\bp
Let $\mathcal J$ be any invariant family of ideals. Since we assume that $I(\mathcal J)$ is an ideal, every $a\in I(\mathcal J)_{+}$ satisfies $\sqrt{a} \in I(\mathcal J)$, and hence $\vartheta_{x, \nu}(a)\in J_{x}$, proving that $I(\mathcal J)_x\subset J_x$ for all $x$. It follows that if $J$ is an ideal in $C^*_\nu(\G)$, then
$$
J\subset I(\{J_x\}_{x\in\Gu})\quad\text{and}\quad J_y=I(\{J_x\}_{x\in\Gu})_y\quad\text{for all}\quad y\in\Gu.
$$
This shows that $(1)$ and $(2)$ are equivalent. Also, obviously $(2)$ implies $(3)$. Finally, if $J=I(\mathcal T)$ for an invariant family of ideals $\mathcal T$, then $J_x\subset T_x$, hence $J\subset I(\{J_x\}_{x})\subset I(\mathcal T)$ and therefore $J=I(\{J_x\}_{x})$. Thus $(3)$ implies $(2)$.
\ep


Let us show now that $I(\mathcal J)$ is indeed an ideal. We will simultaneously establish a connection between the construction of $I(\mathcal J)$ and induced ideals.

Take a point $x\in\Gu$. Given an ideal $J\subset C^*_\nu(\Gxx)$, consider a representation $\pi\colon C^*_\nu(\Gxx)\to B(H)$ with kernel $J$. Define an induced ideal in $C^*_\nu(\G)$ by
$$
\Ind J:=\ker(\Ind\pi\colon C^*_\nu(\G)\to B(\Ind H)).
$$
It is well-known that this definition is independent of the choice of $\pi$. In fact, the construction can equivalently be described as follows.

\begin{lemma}\label{lem:induced-ideal}
Let $\mathcal{B} \subset C^*_\nu(\G)$ be a subset spanning a dense subspace of $C^*_\nu(\G)$. For any ideal $J\subset C^*_\nu(\Gxx)$, we have
$$
\Ind J=\{a\in C^*_\nu(\G): \vartheta_{x,\nu}(b^*a^*ab)\in J\ \text{for all}\ b\in \mathcal{B}\}.
$$
\end{lemma}

\bp
Consider a representation $\pi\colon C^*_\nu(\Gxx)\to B(H)$ with kernel $J$ and let $v\colon\Ind H\to H$ be the coisometry given by $v\xi=\xi(x)$. By Corollary~\ref{cor:contraction}, for all $a,b\in C^*_\nu(\G)$ and $\xi\in H$, we then have
$$
\|(\Ind\pi)(ab)v^*\xi\|^2=(\pi(\vartheta_{x,\nu}(b^*a^*ab))\xi,\xi).
$$
Observe now that $(\Ind\pi)(\mathcal{B})v^* H$ spans a dense set in $\Ind H$. One can see this either from the definitions or by applying~\cite{CN}*{Lemma~1.3}, which shows that if $\pi$ is cyclic with a cyclic vector~$\xi$, then $v^*\xi$ is a cyclic vector for $\Ind\pi$. We conclude that
$$
a\in\ker(\Ind\pi\colon C^*_\nu(\G)\to B(\Ind H))
$$
if and only if $\pi(\vartheta_{x,\nu}(b^*a^*ab))=0$ for all $b \in \mathcal{B}$.
\ep

\begin{cor}\label{cor:ind-primitive}
For any ideals $I\subset C^*_\nu(\G)$ and $J\subset C^*_\nu(\Gxx)$, we have $I\subset\Ind J$ if and only if $\vartheta_{x,\nu}(I)\subset J$.
\end{cor}

We can now describe $\Ind J$ as $I(\mathcal J)$ for a particular family of ideals.

\begin{lemma}\label{lem:induced-ideal2}
Assume $J\subset C^*_\nu(\Gxx)$ is an ideal. Define a family $\mathcal J=(J_y)_y$ by $J_{r(\gamma)}:=\Psi_{\gamma,\nu}(J)$ for $\gamma \in \mathcal{G}_{x}$ and $J_{y}:=C_\nu^{*}(\mathcal{G}_{y}^{y})$ for $y\notin [x]:=r(\G_x)$. Then $\Ind J=I(\mathcal J)$, that is, $a\in\Ind J$ if and only if $\vartheta_{y,\nu}(a^*a)\in J_y$ for all $y\in[x]$.
\end{lemma}

\bp
Assume $a\in\Ind J$. Then by Lemma~\ref{lem:induced-ideal} we have $\vartheta_{x,\nu}(a^*a)\in J_x$. As $\Ind J_x=\Ind J_y$ for all $y\in[x]$, for the same reason we have $\vartheta_{y,\nu}(a^*a)\in J_y$ for all $y\in[x]$, so $\Ind J\subset I(\mathcal J)$.

For the opposite inclusion, assume $a\in I(\mathcal J)$. We need to show that $ \vartheta_{x,\nu}(ba^*ab^*)\in J$ for $b$'s spanning a dense subspace of $C^*_\nu(\G)$. It suffices to consider $b=f\in C_c(W)$ for open bisections~$W$ of~$\G$. If $x\notin r(W)$, then $\vartheta_{x,\nu}(fa^*af^*)=0$. Assume next that $x\in r(W)$ and let $\gamma\in W$ be the unique element such that $r(\gamma)=x$.
A direct computation shows that
$$
\eta_{r(\gamma)}(f*q*f^{*}) = |f(\gamma)|^2(\Ad\gamma)(\eta_{s(\gamma)}(q))\quad \text{for all}\quad q\in \CcG,
$$
so by Lemma~\ref{lem:invariance}
$$
\vartheta_{x,\nu}(fa^*af^*)=|f(\gamma)|^2\Psi_{\gamma,\nu}(\vartheta_{s(\gamma),\nu}(a^*a)).
$$
Hence $\vartheta_{x,\nu}(fa^*af^*)\in\Psi_{\gamma,\nu}(J_{s(\gamma)})=J_x=J$.
\ep

\begin{cor}
For any invariant family of ideals $\mathcal J=(J_x)_{x\in\Gu}$, $I(\mathcal J)$ is an ideal in~$C^*_\nu(\G)$.
\end{cor}

\bp
By the previous lemma we can write $\displaystyle I(\mathcal J)=\bigcap_{x\in\Gu}\Ind J_x$.
\ep

Using the connection with induced ideals we can use known cases of the Effros--Hahn conjecture to give an answer to Question~\ref{ques1} for certain groupoids.

\begin{thm}
Assume $\G$ is an amenable second countable Hausdorff locally compact \'{e}tale groupoid. Then the ideals of $C^*(\G)=C_{r}^{*}(\G)$ are determined by their isotropy fibers.
\end{thm}

\bp
By Lemma~\ref{lem:equiv} we need to show that the ideals $I(\mathcal J)$ exhaust all ideals of $C^*(\G)$. Observe that the collection of ideals $I(\mathcal J)$ is closed under intersections. As every ideal is an intersection of primitive ideals, it is therefore enough to check that every primitive ideal is of the form $I(\mathcal J)$. But this is true, since by \cite{IW} every primitive ideal has the form $\Ind J$ for a primitive ideal $J\subset C^*(\Gxx)$ for some $x$, while $\Ind J$ is of the required form by Lemma~\ref{lem:induced-ideal2}.
\ep

Going beyond the amenable case it is not difficult to give examples where ideals are not determined by their isotropy fibers.

\begin{example}\label{ex:AFS0}
Assume a discrete group $\Gamma$ acts freely on a compact space $X$ and $C(X)\rtimes\Gamma\ne C(X)\rtimes_r\Gamma$. Consider the transformation groupoid $\G=X\rtimes \Gamma$, so that $C^*(\G)=C(X)\rtimes\Gamma$. Let~$I$ be the kernel of the canonical map $C^*(\G)\to C^*_r(\G)=C(X)\rtimes_r\Gamma$. The isotropy groups of $\G$ are trivial and the contractions $\vartheta_x\colon C^*(\G)\to\C$ are the compositions of the canonical conditional expectation $E\colon C(X)\rtimes\Gamma\to C(X)$ with point-evaluations. Since the map $C^*(\G)\to C^*_r(\G)$ intertwines $E$ and the faithful canonical conditional expectation $C(X)\rtimes_r\Gamma\to C(X)$, we get $I \subset \ker E$. Hence the isotropy fibers of $I$ are all zero (this also follows immediately from Proposition~\ref{prop:U(I)} below), yet $I\ne0$.\ee
\end{example}

This example may not look very interesting, since in the reduced C$^*$-algebra $C^*_r(\G)=C(X)\rtimes_r\Gamma$ as above the ideals \emph{are} determined by their isotropy fibers, at least when $\Gamma$ is exact, see Proposition~\ref{prop:U(I)} and Remark~\ref{rem:Renault}. But similar phenomena also show up for reduced groupoid C$^*$-algebras.

\begin{example}\label{ex:AFS}
In \cite{AFS}, Alekseev and Finn-Sell construct a principal Hausdorff locally compact \'{e}tale groupoid $\G$ together with a compact invariant subset $X\subset\Gu$ satisfying the following properties: (a) the reduced and maximal C$^*$-norms on $C_c(\G)$ coincide; (b) $\G_X=X\rtimes\Gamma$ for a free action of a discrete group $\Gamma$ on $X$ such that $C(X)\rtimes\Gamma\ne C(X)\rtimes_r\Gamma$. Consider the ideal $J:=C^*_r(\G_{X^c})=C^*(\G_{X^c})$ in $C^*_r(\G)=C^*(\G)$, and let $I$ be the kernel of the composition of $C^*(\mathcal{G}) \to C^*(\G_X)$ with $C^*(\G_X)\to C^*_r(\G_X)$. Then $J\subset I$ and $J\ne I$, since
$$
0 \to C^{*}(\mathcal{G}_{X^c}) \to C^{*}(\mathcal{G}) \to C^*(\G_X)\to 0
$$
is exact. Then, similarly to the previous example, we have $J_x=I_x=\C$ for $x\in X^c$ and $J_x=I_x=0$ for $x\in X$, yet $J\ne I$. \ee
\end{example}

Returning to general locally compact \'{e}tale groupoids, we will next describe when isotropy fibers are proper ideals. For an ideal $I\subset C^*_\nu(\G)$, let $U(I)\subset\Gu$ be the open, possibly empty, subset such that
$$
C_0(\Gu)\cap I=C_0(U(I)).
$$
It is well-known and easy to check that the set $U(I)$ is invariant. It is also clear that $U(I)=\Gu$ if and only if $I=C^*_\nu(\G)$.

\begin{prop}\label{prop:U(I)}
Assume $\G$ is a locally compact \'{e}tale groupoid and $C^*_\nu(\G)$ is a groupoid C$^*$-algebra. Then, for any ideal $I\subset C^*_\nu(\G)$, we have
$$
U(I)=\{x\in\Gu: I_x=C^*_\nu(\Gxx)\}.
$$
\end{prop}

\bp
If $f\in C_0(U(I))=C_{0}(\Gu)\cap I$, then $f(x)1\in\vartheta_{x , \nu}(I)$. Therefore we only need to prove that $\vartheta_{\nu,x}(I)$ is a proper ideal for all $x\in U(I)^c$.

We use an idea from~\cite{AS}*{Theorem~1}, see also \citelist{\cite{BKKO}*{Lemma~7.9}\cite{BNRSW}*{Theorem~3.2}\cite{Ka}*{Lemma~2.4}} for its recent similar applications. Fix a point $x\in U(I)^c$ and consider the state $\psi$ on the C$^*$-subalgebra $I+C_0(\Gu)\subset C^*_\nu(\G)$ defined as the composition of the homomorphisms
$$
I+C_0(\Gu)\to (I+C_0(\Gu))/I\cong C_0(\Gu)/(C_0(\Gu)\cap I)\cong C_0(U(I)^c)
$$
with the point-evaluation $\ev_x\colon C_0(U(I)^c)\to\C$. Extend it to a state on $C^*_\nu(\G)$. By Proposition~\ref{prop:pointeval} this extension must be of the form $\varphi\circ\vartheta_{x, \nu}$ for a state $\varphi$ on $C^*_\nu(\Gxx)$. As $\psi(I)=0$ by construction, we have $\varphi(\vartheta_{x,\nu}(I))=0$, hence $\vartheta_{x,\nu}(I)\ne C^*_\nu(\Gxx)$.
\ep

We are now ready to prove Theorem~\ref{thm:A}. We remind of the statement.

\begin{thm}\label{thm:induction}
Assume $\G$ is a locally compact \'{e}tale groupoid, $C^*_\nu(\G)$ is a groupoid C$^*$-algebra and $I\subset C^*_\nu(\G)$ is a proper ideal. Then, for every point $x\in U(I)^c$, there is a primitive ideal $J\subset C^*_\nu(\Gxx)$ such that $I\subset\Ind J$.
\end{thm}

\bp
Take any $x\in U(I)^c$. By Proposition~\ref{prop:U(I)}, $\vartheta_{x,\nu}(I)$ is a proper ideal in $C^*_\nu(\Gxx)$, hence it is contained in a primitive ideal $J\subset C^*_\nu(\Gxx)$. Then $I\subset\Ind J$ by Corollary~\ref{cor:ind-primitive}.
\ep

We end this section with a general remark that will play no further role in the paper.

\begin{remark}
We have described the intersections of induced ideals as ideals arising from invariant families of ideals in $C^*_\nu(\Gxx)$. But there is a third point view on them: they are exactly the ideals in $C_{\nu}^{*}(\G)$ that can be obtained as intersections of elements of
$$
\mathcal{M}:=\{ L \subset C_{\nu}^{*}(\G) \; | \; \text{$L$ is a maximal closed left ideal and $L \, C_{0}(\Gu) \subset L$} \} \; .
$$

Let us sketch why this is true. Recall that there is a bijection $\varphi \mapsto \mathcal{L}_{\varphi}:=\{a\in \mathcal{A} \; | \; \varphi(a^{*}a)=0 \}$ between the pure states and the maximal closed left ideals of a C$^{*}$-algebra $\mathcal{A}$, and that any closed proper left ideal of $\mathcal{A}$ is an intersection of maximal closed left ideals. If $\varphi$ is a pure state on~$C_{\nu}^{*}(\Gxx)$, then $\varphi \circ \vartheta_{x , \nu}$ is a pure state on $C_{\nu}^{*}(\G)$  by Proposition \ref{prop:pointeval} and clearly $\mathcal{L}_{\varphi \circ \vartheta_{x,  \nu}} \in \mathcal{M}$.

We claim that all elements of $\mathcal{M}$ are of this form. So let $\psi$ be a pure state on $C_{\nu}^{*}(\G)$ with $\mathcal{L}_{\psi} \in \mathcal{M}$ and let $f\in C_{c}(\Gu)$. By assumption the functional $\psi(f^{*} \cdot f)$ satisfies $\mathcal{L}_{\psi} \subset \mathcal{L}_{\psi(f^{*} \cdot f)}$. Hence $\mathcal{L}_{\psi(f^{*} \cdot f)}= \mathcal{L}_{\psi}$ or $\mathcal{L}_{\psi(f^{*} \cdot f)}=C_{\nu}^{*}(\G)$, in either case it follows that $\psi(f^{*}qf)=\psi(f^{*}f)\psi(q)$ for $q\in C_{c}(\Gu)$, and therefore $\psi|_{C_{0}(\Gu)}= \ev_{x}$ for some $x\in \Gu$. By Proposition \ref{prop:pointeval} it follows that $\psi=\varphi \circ \vartheta_{x, \nu}$ for a pure state $\varphi$ on $C_{\nu}^{*}(\Gxx)$, proving the claim.

Having described the elements of $\mathcal{M}$, one can use \eqref{eq:I(J)} and the observation that $\mathcal{L}_{\varphi \circ \vartheta_{x, \nu}}=\{ a\in C_{\nu}^{*}(\G) \; | \;  \vartheta_{x, \nu}(a^{*}a)\in \mathcal{L}_{\varphi}\}$, to prove that an ideal $J\subset C_{\nu}^{*}(\G)$ satisfies
$
J=\bigcap_{L\in \mathcal{M}, \; J\subset L} L
$
if and only if there exists an invariant family $\mathcal{J}$ of ideals with $J=I(\mathcal{J})$.
\end{remark}

\bigskip

\section{Maximal ideals} \label{sectionMaximal}

We continue to assume that $\G$ is a locally compact \'etale groupoid and $C^*_\nu(\G)$ is a groupoid C$^*$-algebra. Although in general ideals of $C^*_\nu(\G)$ are not determined by their isotropy fibers, in this section we will show that it is possible to completely classify the maximal ideals in terms of their fibers.

\smallskip
We will need the following definition. Recall that $(\Ind J)_x\subset J$ for any ideal $J\subset C^*_\nu(\Gxx)$.

\begin{defn}
We say that an ideal $J\subset C^*_\nu(\Gxx)$ is \emph{$\G$-invariant}, if $J=(\Ind J)_x$, that is,
$$
J=\vartheta_{x,\nu}(\Ind J).
$$
\end{defn}

\begin{lemma}\label{lem:G-invariance}
For any ideal $I\subset C^*_\nu(\G)$ and any $x\in\Gu$, the ideal $I_x\subset C^*_\nu(\Gxx)$ is $\G$-invariant.
\end{lemma}

\bp
For $J:=I_x$ we have $(\Ind J)_x\subset J$ and $I\subset\Ind J$ by Corollary~\ref{cor:ind-primitive}, whence $J=I_x\subset (\Ind J)_x$ and therefore $(\Ind J)_x=J$.
\ep

If $(I_i)_i$ is a chain of proper $\G$-invariant ideals, then $\overline{\bigcup_i I_i}$ is proper and $\G$-invariant. A standard application of Zorn's lemma shows then that the nonempty set of proper $\G$-invariant ideals in~$C^*_\nu(\Gxx)$ contains maximal elements, which we call \emph{maximal $\G$-invariant ideals}.

\smallskip

We will also need the following known property of induced ideals.

\begin{lemma}\label{lem:orbit}
If $J\subset C^*_\nu(\Gxx)$ is a proper ideal, then $U(\Ind J)^c$ is the closure of the orbit $[x]$ of $x$.
\end{lemma}

\bp
Using Lemma~\ref{lem:induced-ideal2} we see that $f\in\Ind J$ if and only if $|f(y)|^21\in J_y$ for all $y\in[x]$, that is, if and only if $f$ vanishes on $\overline{[x]}$.
\ep

We are now ready to parameterize the maximal ideals. The following is a more precise version of Theorem~\ref{thm:B}.

\begin{thm}\label{thm:maximal}
Assume $\G$ is a locally compact \'etale groupoid and $C^*_\nu(\G)$ is a groupoid C$^*$-algebra. For every minimal closed invariant subset $F\subset\Gu$, choose a point $x_F\in F$. Then the map $I\mapsto (U(I)^c,I_{x_{U(I)^c}})$ is a bijection between the set of maximal ideals in $C^*_\nu(\G)$ and the set of pairs $(F,J)$ consisting of a minimal closed invariant subset $F\subset\Gu$ and a maximal $\G$-invariant ideal $J\subset C^*_\nu(\G^{x_F}_{x_F})$. The inverse map is given by $(F,J)\mapsto\Ind J$.
\end{thm}

Recall that a closed invariant subset $F\subset \Gu$ is minimal if and only if $\overline{[x]}=F$ for all $x\in F$.
Note that existence of minimal closed invariant subsets and maximal ideals is not guaranteed unless $\Gu$ is compact. So if for a noncompact $\Gu$ either the set of maximal ideals in $C^*_\nu(\G)$ or the set of minimal closed invariant subsets of $\Gu$ is empty, then the theorem says that the other set is empty as well.

\bp[Proof of Theorem~\ref{thm:maximal}]
Let $I$ be a maximal ideal. Take any $x\in U(I)^c$. By Proposition~\ref{prop:U(I)}, the ideal $I_x$ is proper. Then $\Ind I_x$ is a proper ideal in $C^*_\nu(G)$ that contains $I$ by Corollary~\ref{cor:ind-primitive}. Hence $I=\Ind I_x$ by maximality of $I$. Then, by Lemma~\ref{lem:orbit}, $U(I)^c=\overline{[x]}$. Since $x$ was arbitrary, this shows that $F:=U(I)^c$ is a minimal closed invariant subset of $\Gu$.

Similarly, if $J$ is any proper $\G$-invariant ideal containing $I_x$, then $I=\Ind J$ and hence $I_x=J$. Therefore $I_x$ is a maximal $\G$-invariant ideal for any $x\in U(I)^c$. Hence the map $I\mapsto (F,I_{x_F})$ in the statement of the theorem is well-defined. As $I=\Ind I_{x_F}$, this map is injective.

It remains to check surjectivity. Assume $F\subset\Gu$ is a minimal closed invariant subset and $J\subset C^*_\nu(\G^{x_F}_{x_F})$ is a maximal $\G$-invariant ideal. Consider $I:=\Ind J$. Then $I_{x_F}=J$ by $\G$-invariance of $J$ and $U(I)^c=\overline{[x_F]}=F$ by minimality of $F$. Therefore to finish the proof it remains to show that $I$ is maximal.

Assume $I\subset H$ for a proper ideal $H$. Then $\emptyset\ne U(H)^c\subset U(I)^c=F$, hence $U(H)^c=F$ by minimality of $F$. By Proposition~\ref{prop:U(I)}, the ideal $H_{x_F}$ is proper. By Lemma~\ref{lem:G-invariance} it is also $\G$-invariant. As $J=I_{x_F}\subset H_{x_F}$, we get $H_{x_F}=J$ by maximality of $J$. Then $I\subset H\subset \Ind H_{x_F}=\Ind J=I$. In conclusion $I=H$, proving that $I$ is maximal.
\ep

In practice it may not be easy to describe maximal $\G$-invariant ideals. We will give an example where we can describe the maximal $\G$-invariant ideals in Corollary~\ref{cor:max}.

\begin{remark}
The maximal ideals are primitive, but maximal $\G$-invariant ideals are in general not. It is possible to formulate the above parameterization in terms of primitive ideals as follows. Recall that the closed subsets of $\Prim C^*_\nu(\Gxx)$ in the hull-kernel topology have the form $V(J)=\{I\in \Prim C^*_\nu(\Gxx): J\subset I\}$ for ideals $J\subset C^*_\nu(\Gxx)$. We have a map on the set $\mathrm{Cl}(\Prim C^*_\nu(\Gxx))$ of closed subsets of $\Prim C^*_\nu(\Gxx)$ defined by $V(J)\mapsto V((\Ind J)_x)$. Let us say that a closed subset of $\Prim C^*_\nu(\Gxx)$ is $\G$-invariant if it is a fixed point of this map. The map $J\mapsto V(J)$ gives a one-to-one correspondence between the $\G$-invariant ideals in $C^*_\nu(\Gxx)$ and the $\G$-invariant closed subsets of $\Prim C^*_\nu(\Gxx)$. Then we can say that the maximal ideals of $C^*_\nu(\G)$ are parameterized by the pairs $(F,V)$ consisting of a minimal closed invariant subset $F\subset\Gu$ and a minimal $\G$-invariant closed subset $V\subset \Prim C^*_\nu(\G^{x_F}_{x_F})$. The maximal ideal corresponding to $(F,V)$ is $\Ind J$, where $J$ is any element of $V$.
\end{remark}

\bigskip

\section{Simplicity, the ideal intersection property and a uniqueness theorem}\label{sec:simplicity}

In this section we apply our results to study the question whether every nontrivial ideal in~$C^*_r(\G)$ intersects nontrivially with $C^*_r(\HH)$ for an open subgroupoid $\HH\subset\G$.

\smallskip

We start with the case $\HH=\Gu$. Then, if the above question has positive answer, $\G$ is said to have the \emph{ideal intersection property}. On the one hand, the results we are going to prove about this property generalize a number of earlier results, in particular, some recent results on transformation groupoids. On the other hand, they are not going to be as sharp as the very recent results of Kennedy--Kim--Li--Raum--Ursu~\cite{KKLRU}. The main reason to include them is that they can still provide a quicker and more elementary route to properties that are important in concrete examples.

\smallskip

Recall that a subset of a topological space is called nowhere dense if its closure has empty interior. A subset is called meager, if it is a countable union of nowhere dense sets, and the complement of a meager set is by definition a residual set. 

From now on, we view $\rho_x$ as a representation of $C^*_r(\G)$. Following~\cite{MR4246403}, define the \emph{singular ideal} of $C^*_r(\G)$ by
$$
J_\sing:=\{a\in C^*_r(\G)\mid\ \text{the set of}\ x\in \G^{(0)}\ \text{such that}\ \rho_x(a)\delta_x\ne0\ \text{is meager}\},
$$
and the \emph{essential groupoid C$^*$-algebra} of $\G$ by
$$
C^*_\ess(\G):=C^*_r(\G)/J_\sing.
$$
If $\G$ is Hausdorff, then $J_\sing=0$ and $C^*_\ess(\G)=C^*_r(\G)$. By \cite{NS}*{Proposition~1.12}, if $\G$ can be covered by countably many open bisections, then
$$
J_\sing=\{a\in C^*_r(\G)\mid\ \text{the set of}\ x\in \G^{(0)}\ \text{such that}\ \rho_x(a)\ne0\ \text{is meager}\}.
$$

We will need the following observation, which is well-known in the Hausdorff case.

\begin{lemma} \label{lem:countint}
Assume $\G$ is a locally compact \'etale groupoid such that either $\G$ is Hausdorff or~$\G$ can be covered by countably many open bisections. Then, for any dense subset $X\subset\Gu$, we have
\begin{equation*}
\bigcap_{x\in X}\ker\rho_x\subset J_\sing.
\end{equation*}
\end{lemma}

\bp
When $\G$ can be covered by countably many open bisections, this follows from \cite{NS}*{Proposition~1.12} and its proof. In the Hausdorff case the canonical conditional expectation $E\colon C_{r}^{*}(\G)\to C_{0}(\Gu)$ is faithful and $E(a)(x)=(\rho_x(a)\delta_x,\delta_x)$ for all $x\in \Gu$. Therefore $E(a^*a)(x)=0$ for all $x\in X$ and $a\in \bigcap_{x\in X}\ker\rho_x$, so by density of~$X$ in~$\Gu$ we get~$\bigcap_{x\in X}\ker\rho_x=0$.
\ep

To describe when the inequality in Lemma~\ref{lem:countint} is an equality, recall from~\cite{MR4246403} that $x\in \Gu$ is called a \emph{dangerous point} if there exists a net in $\Gu$ which both converges to $x$ and to an element of $\Gxx\setminus\{x\}$. Let $D\subset \Gu$ denote the set of dangerous points. This set is empty if and only if $\G$ is Hausdorff. By~\cite{MR4246403}*{Lemma~7.15}, if $\G$ can be covered by countably many open bisections, then the set~$D$ is meager, and then by~\cite{NS}*{Proposition~1.12} we have
\begin{equation}\label{eq:sing2}
\bigcap_{x\in X}\ker\rho_x = J_\sing
\end{equation}
for any dense subset $X\subset D^c$. In fact, the set $D$ here can be replaced by a smaller set of \emph{extremely dangerous points}~\cite{NS}, which can be empty even for non-Hausdorff groupoids, but this is not going to be important for us.

\smallskip

As in~\cite{CN}, denote by $\vartheta_{x,r}\colon C^*_r(\G)\to C^*_r(\Gxx)$ the canonical contraction. It coincides with the composition of $\vartheta_{x,e}\colon C^*_r(\G)\to C^*_e(\Gxx)$ with the quotient map $C^*_e(\Gxx)\to C^*_r(\Gxx)$. For an ideal $I\subset C^*_r(\G)$ and $x\in\Gu$, put $I_{x,r}:=\vartheta_{x,r}(I)$. Thus, $I_{x,r}$ is a quotient of $I_x$, so it is an ideal in~$C^*_r(\Gxx)$.

\begin{lemma}\label{lem:Jsing}
Assume $\G$ is a locally compact \'etale groupoid such that either $\G$ is Hausdorff or~$\G$ can be covered by countably many open bisections. Assume $I\subset C^*_r(\G)$ is an ideal such that the set $X:=\{x\in\Gu\mid I_{x,r}=0\}$ is dense in $\Gu$. Then $I\subset J_\sing$.
\end{lemma}

\bp
As $\Ind\lambda_{\Gxx}\sim\rho_x$, by Corollary~\ref{cor:ind-primitive} the condition $I_{x,r}=0$ means that $I\subset\ker\rho_x$, and hence it follows from Lemma \ref{lem:countint} that $I\subset J_\sing$.
\ep

The following theorem and corollary generalize \citelist{\cite{O}*{Theorem~14}\cite{Ka}*{Lemma~2.4(i)}\cite{KS}*{Proposition~6.6}(i)}. For the topologically principal groupoids this is also closely related to~\cite{MR3988622}*{Theorem~4.10(3)}.

\begin{thm}\label{thm:ideal-intersection}
Assume $\G$ is a locally compact \'etale groupoid such that either $\G$ is Hausdorff or~$\G$ can be covered by countably many open bisections. Assume that the set
$$
X:=\{x\in\Gu\mid \Gxx\ \text{is}\ C^*\text{-simple and}\ C^*_e(\Gxx)=C^*_r(\Gxx)\}
$$
is dense in $\Gu$. Then we have $I\subset J_\sing$ for any ideal $I\subset C^*_r(\G)$ with $U(I)=\emptyset$.
\end{thm}

\bp
Assume $I\subset C^*_r(\G)$ is an ideal such that $U(I)=\emptyset$. By Proposition~\ref{prop:U(I)} it follows that $I_x\subset C^*_e(\Gxx)$ is a proper ideal for all $x\in\Gu$. By our assumptions this implies that $I_x=0$ for all $x\in X$. Then $I\subset J_\sing$ by Lemma~\ref{lem:Jsing}.
\ep

The most unsatisfactory part of the assumptions of the theorem is of course the equality $C^*_e(\Gxx)=C^*_r(\Gxx)$. We remind, however, that by~\cite{CN} it is satisfied for the transformation-type groupoids $X\rtimes\Gamma$ defined by partial actions of discrete groups, as well as for the groupoids injectively graded by exact discrete groups. Recall that $\G$ is said to be injectively graded by $\Gamma$, if we are given a continuous homomorphism (or $1$-cocycle) $\Phi\colon\G\to\Gamma$ such that $\Phi|_{\Gxx}$ is injective for all $x\in\Gu$.

\begin{cor}\label{cor409}
Assume $\G$ is a minimal locally compact \'etale groupoid such that either $\G$ is Hausdorff or~$\G$ can be covered by countably many open bisections. Assume there exists a point $x\in\Gu$ such that $\Gxx$ is C$^*$-simple and $C^*_e(\Gxx)=C^*_r(\Gxx)$. Then $J_{\sing}=\ker \rho_{x}$ and $J_\sing$ is the largest proper ideal in $C^*_r(\G)$. In particular, $C_{\ess}^{*}(\G)$ is simple.
\end{cor}

\bp
By the minimality assumption we have $U(I)=\emptyset$ for every proper ideal $I\subset C^*_r(\G)$. Since $C^*_e(\Gxx)$ is simple, Proposition~\ref{prop:U(I)} and Corollary~\ref{cor:ind-primitive}
implies that $I\subset\ker\rho_x$. As~$[x]$ is dense in $\Gu$ and $\rho_y\sim\rho_x$ for all $y\in[x]$ we get by Lemma~\ref{lem:countint} that $I\subset\ker\rho_x\subset J_\sing$. Since this also applies to $I=J_\sing$, it follows that $J_\sing=\ker\rho_x$ and $J_\sing$ is the largest proper ideal in $C^*_r(\G)$.
\ep

This corollary leads to the following question that we have been unable to answer.

\begin{question} \label{que:denseexo}
Is there a groupoid $\G$ with a unit $x\in\Gu$ such that $[x]$ is dense in~$\Gu$ and $C^*_e(\Gxx)\ne C^*_r(\Gxx)$? Is there a minimal and Hausdorff groupoid with this property?
\end{question}

\begin{remark}\label{rem:Renault}
One can use Theorem \ref{thm:ideal-intersection} to provide another proof of \cite{Rbook}*{Proposition~II.4.6}. The statement of this proposition misses one assumption and in the correct form should read as follows: if $\G$ is an inner exact Hausdorff locally compact \'etale groupoid, with the property that for any closed invariant set $X\subset \Gu$ the subset of points in $X$ with trivial isotropy is dense in~$X$, then the map $U\mapsto C_{r}^{*}(\G_{U})$ is a bijection from the open invariant subsets of $\Gu$ onto the ideals of~$C_{r}^{*}(\G)$. To see this, notice that if $I\subset C_{r}^{*}(\G)$ is a proper ideal, then $C^*_r(\G_{U(I)})\subset I$ and by inner exactness we get an ideal $J\cong I/C^*_r(\G_{U(I)})$ in $C_{r}^{*}(\G_{U(I)^{c}})$ such that $U(J)=\emptyset$. Then $J=0$ by Theorem \ref{thm:ideal-intersection}, hence $I=C^*_r(\G_{U(I)})$.\ee
\end{remark}

Turning to other open subgroupoids, we start with the following technical result, which generalizes Proposition~\ref{prop:U(I)} and is of independent interest.

\begin{prop}\label{prop:intersection}
Assume $\G$ is a locally compact \'etale groupoid, $C^*_\nu(\G)$ is a groupoid C$^*$-algebra, $\HH\subset\G$ is an open subgroupoid and $C^*_\nu(\HH)$ is the closure of $C_c(\HH)$ in $C^*_\nu(\G)$. Assume $I\subset C^*_\nu(\G)$ is an ideal and let $J:=C^*_\nu(\HH)\cap I$. Then $J_x=C^*_\nu(\HH^x_x)\cap I_x$ for all $x\in\HH^{(0)}$.
\end{prop}

Recall that by a discussion at the end of Section~\ref{ssec:completions}, the norm on~$\C\HH^x_x$ is the restriction of the norm on~$ \C\Gxx$, so that we can view $C^*_\nu(\HH^x_x)$ as a subalgebra of~$C^*_\nu(\Gxx)$.

\bp Fix a point $x\in\HH^{(0)}$. We obviously have $J_x\subset C^*_\nu(\HH^x_x)\cap I_x$. In order to prove the equality it suffices to show that every state $\varphi$ on $C^*_\nu(\Gxx)$ that vanishes on~$J_x$ must vanish on~$C^*_\nu(\HH^x_x)\cap I_x$. Fix such a state $\varphi$ and consider the state $\psi:=\varphi\circ\vartheta_{x,\nu}$. Its restriction to~$C^*_\nu(\HH)$ vanishes on~~$J$ and can therefore be considered as a state on $C^*_\nu(\HH)/J$. Define a state on the C$^*$-subalgebra $I+C^*_\nu(\HH)\subset C^*_\nu(\G)$ as the composition
$$
I+C^*_\nu(\HH)\to (I+C^*_\nu(\HH))/I\cong C^*_\nu(\HH)/J\xrightarrow{\psi}\C,
$$
and then extend it to a state $\psi'$ on $C^*_\nu(\G)$. Since the state $\ev_x$ on $C_0(\Gu)$ is the unique extension of $\ev_x|_{C_0(\HH^{(0)})}$, by Proposition~\ref{prop:pointeval} we have $\psi'=\varphi'\circ\vartheta_{x,\nu}$ for a state $\varphi'$ on $C^*_\nu(\Gxx)$. But then $\varphi'=\varphi$ on $C^*_\nu(\HH^x_x)$.  As $\psi'$ vanishes on $I$ by construction, the state $\varphi'$ vanishes on $I_x$ and hence~$\varphi$ vanishes on $C^*_\nu(\HH^x_x)\cap I_x$.
\ep

The following result is a generalization of \citelist{\cite{BNRSW}*{Theorem~3.1(b)} \cite{S}*{Theorem~2.1(ii)}}, which are in turn inspired by a \emph{uniqueness theorem} for $k$-graph C$^*$-algebras.

\begin{thm}
Assume $\G$ is a locally compact \'etale groupoid such that either $\G$ is Hausdorff or~$\G$ can be covered by countably many open bisections. Assume $\HH\subset\G$ is an open subgroupoid such that the set
$
X:=\{x\in\HH^{(0)}\mid \HH^x_x=\Gxx\}
$
is dense in $\Gu$. Then we have $I\subset J_\sing$ for any ideal $I\subset C^*_r(\G)$ such that $C^*_r(\HH)\cap I=0$.
\end{thm}

\bp
Assume $I\subset C^*_r(\G)$ is such that $C^*_r(\HH)\cap I=0$. By the previous proposition we then get $I_x=0$ for all $x\in X$. By Lemma~\ref{lem:Jsing} we conclude that $I\subset J_\sing$.
\ep

The following corollary generalizes~\cite{MR4043708}*{Theorem~7.7}.

\begin{cor}
Assume $\G$ is a minimal Hausdorff locally compact \'etale groupoid and $\HH\subset\G$ is an open subgroupoid such that its isotropy bundle is normal in $\G$. Assume there exist points $x,y\in\Gu$ such that $\HH^x_x=\Gxx$, $\HH^y_y$ is C$^*$-simple and $C^*_e(\HH^y_y)=C^*_r(\HH^y_y)$. Then $C_r^{*}(\G)$ is simple.
\end{cor}

The normality assumption here means that $\HH^{(0)}=\Gu$ and $\gamma\HH^{s(\gamma)}_{s(\gamma)}\gamma^{-1}=\HH^{r(\gamma)}_{r(\gamma)}$ for all~$\gamma\in\G$.

\bp
Consider the set $X:=\{z\in\Gu\mid \HH^z_z=\G^z_z\}$. As $\HH^x_x=\Gxx$, this set contains the $\G$-orbit of~$x$ by normality of the isotropy bundle of~$\HH$. Since this orbit is dense in $\Gu$ by minimality of~$\G$, by the previous theorem it is therefore enough to show that if $I\subset C^*_r(\G)$ is a proper ideal then $C^*_r(\HH)\cap I=0$.

Consider the ideal $J:=C^*_r(\HH)\cap I\subset C^*_r(\HH)$. As $I$ is a proper ideal, $U(I)=\emptyset$ by minimality of $\G$. By Proposition~\ref{prop:U(I)} it follows that $I_z\subset C^*_e(\G^z_z)$ is a proper ideal for all $z\in\Gu$. Hence $J_z\subset C^*_e(\HH^z_z)$ is a proper ideal for all $z$. By our assumptions it follows that $J_y=0$.

Observe next that if $z\in\Gu$ and $\gamma\in\G_z$, then by Proposition~\ref{prop:intersection} and Lemma~\ref{lem:invariance} we have
$$
J_{r(\gamma)}=C^*_e(\HH^{r(\gamma)}_{r(\gamma)})\cap I_{r(\gamma)}=\Psi_{\gamma,e}(C^*_e(\HH^z_z)\cap I_z)=\Psi_{\gamma,e}(J_z).
$$
It follows that $J_z=0$ for all $z$ in the $\G$-orbit of $y$, which is dense in $\Gu$ by minimality of $\G$. By Lemma~\ref{lem:Jsing} we conclude that $J=0$.
\ep

\bigskip

\section{Exotic norms and fibers of the singular ideal} \label{sec:exoticfiber}

It is an important problem to understand when  $C_{e}^{*}(\Gxx)\ne C_{r}^{*}(\Gxx)$. The goal of this section is to prove the following theorem, which generalizes~\cite{MR4043708}*{Theorem~4.12} and shows that, at least in the Hausdorff case, this does not happen too often.

\begin{thm} \label{thm:singfiber}
Assume $\G$ is a locally compact \'etale groupoid that can be covered by countably many open bisections. Then the set
$$
\{x\in\Gu\mid J_{\sing,x}=\ker(C^*_e(\Gxx)\to C^*_r(\Gxx))\}
$$
is residual in $\Gu$. In particular, if $J_\sing=0$, then the set $\{x\in\Gu\mid C^*_e(\Gxx)=C^*_r(\Gxx)\}$ is residual.
\end{thm}

As we have already observed, $J_\sing=0$ if $\G$ is Hausdorff. Let us remark that $J_\sing$ is often zero even for non-Hausdorff groupoids. Some sufficient conditions for this are given in~\cite{MR3988622}*{Lemma~4.11} and \cite{NS}*{Proposition~1.12}.

\smallskip

We will first reduce the proof of the theorem to the case of group bundles. Denote by $\IsoG$ the isotropy bundle $\{g\in\G:s(g)=r(g)\}$ of $\G$. Its interior $\Iso$ is an open subgroupoid of $\G$ with unit space $\Gu$.

\begin{lemma}\label{lem:iso-points}
For any locally compact \'etale groupoid $\G$ that can be covered by countably many open bisections, the set of points $x\in\Gu$ such that $\Iso_x=\Gxx$ is residual.
\end{lemma}

\bp
This is essentially \cite{BNRSW}*{Lemma 3.3(a)}, but the result is formulated in a weaker form there. The lemma follows by observing that if $W\subset \G$ is an open bisection and $T\colon s(W)\to r(W)$ is the homeomorphism defined by $W$, then $W$ contributes neither to $\Gxx$ nor to $\Iso_x$ if $x\notin X:= \{y\in s(W): y=Ty\}$,  $W$ contributes to both isotropy groups if $x\in X^\circ$, and finally the set $X\setminus X^{\circ}$ is nowhere dense in $\Gu$, since it is locally closed and has empty interior.
\ep

By the definition of the singular ideals it is clear that the singular ideal of $C^*_r(\Iso)$ is $C^*_r(\Iso)\cap J_\sing$. By Proposition~\ref{prop:intersection} its isotropy fibers are therefore $C^*_e(\Iso_x)\cap J_{\sing,x}$. Since the set of points~$x$ such that $\Iso_x=\Gxx$ is residual by the previous lemma, we conclude that it suffices to prove Theorem \ref{thm:singfiber} for $\Iso$ instead of~$\G$. Assume therefore from now on that $\G$ is a group bundle.

\smallskip

Before continuing the proof, we will need some technical observations regarding the fibers of~$J_\sing$ for group bundles. Recall from Remark \ref{rem:groupbundle} that $C^*_r(\G)$ is a $C_0(\Gu)$-algebra and $C^*_e(\Gxx)$ is its fiber at $x$. Consider also the $C_0(\Gu)$-algebra $C^*_\ess(\G)$ and its fibers $C^*_\ess(\G)_x$. The short exact sequence
$$
0\to J_\sing\to C^*_r(\G)\to C^*_\ess(\G)\to0
$$
of $C_0(\Gu)$-algebras gives rise to short exact sequences
\begin{equation}\label{eq:exact-fiber}
0\to J_{\sing,x}\to C^*_e(\Gxx)\to C^*_\ess(\G)_x\to0.
\end{equation}

Denote by $\pi\colon C^*_r(\G)\to C^*_\ess(\G)$ the quotient map. For $a\in C^*_r(\G)$ write $\pi(a)_x$ for the image of $\pi(a)$ in $C^*_\ess(\G)_x$. For every $a\in C^*_r(\G)$, define two functions $M_a$ and $m_a$ on $\Gu$ by
$$
M_a(x):=\|\pi(a)_x\|,\qquad m_a(x):=\begin{cases}\|\rho_x(a)\|,& \text{if}\ x\in D^c,\\
\liminf_{D^c\ni y\to x}\|\rho_y(a)\|,& \text{if}\ x\in D,\end{cases}
$$
where $D$ denotes the meager set of dangerous points. Similarly to the proof of~\cite{MR4043708}*{Theorem~4.12}, these functions have the following properties.

\begin{lemma} \label{lem:maMa}
For every $a\in C^*_r(\G)$, we have:
\begin{enumerate}
\item $M_a$ is upper semicontinuous;
\item $m_a$ is lower semicontinuous;
\item $m_a\le M_a$;
\item $\displaystyle\sup_{x\in U}m_a(x)=\sup_{x\in U}M_a(x)$ for every open subset $U\subset\Gu$.
\end{enumerate}
\end{lemma}

\bp
(1) This is a general property of $C_0(\Gu)$-algebras.

\smallskip

(2) By \cite{NS}*{Proposition~1.4}, if $x\in D^c$ and $x\in\bar Y$ for a subset $Y\subset\Gu$, then $\rho_x$ is weakly contained in $\bigoplus_{y\in Y}\rho_y$.
It follows that if $x\in D^c$, then
$$
\|\rho_x(a)\|\le \liminf_{y\to x}\|\rho_y(a)\|,\quad\text{hence}\quad m_a(x)\le \liminf_{D^c\ni y\to x}m_a(y).
$$
Therefore $m_a$ is lower semicontinuous on $D^c$. By the definition of $m_a$ on $D$, this implies that~$m_a$ is lower semicontinuous on $\Gu$.

\smallskip

(3) If $x\in D^c$, then by \eqref{eq:sing2} we have $J_\sing\subset\ker\rho_x$. Hence $J_{\sing,x}\subset \ker(C^*_e(\Gxx)\to C^*_r(\Gxx))$. By exactness of~\eqref{eq:exact-fiber} we get
$$
m_a(x)=\|\rho_x(a)\|\le\|\pi(a)_x\|=M_a(x).
$$
As $m_a$ is lower semicontinuous and $M_a$ is upper semicontinuous, we conclude that $m_a(x)\le M_a(x)$ for all $x\in\Gu$.

\smallskip

(4) Let us first check this for $U=\Gu$. By \eqref{eq:sing2}, the direct sum of the representation $\rho_x$ for $x\in D^c$ defines a faithful representation of $C^*_\ess(\G)$. Hence
$$
\sup_{x\in\Gu}m_a(x)=\sup_{x\in D^c}m_a(x)=\sup_{x\in D^c}\|\rho_x(a)\|=\|\pi(a)\|.
$$
On the other hand,
$$
\sup_{x\in\Gu}M_a(x)=\sup_{x\in\Gu}\|\pi(a)_x\|=\|\pi(a)\|
$$
by properties of $C_0(\Gu)$-algebras.

For arbitrary $U$, by (3) we obviously have $\sup_{x\in U}m_a(x)\le\sup_{x\in U}M_a(x)$. To prove the equality, fix $\eps>0$ and choose $y\in U$ such that
$$
M_a(y)>\sup_{x\in U}M_a(x)-\eps.
$$
Let $f\in C_c(\Gu)$ be a function such that $\supp f\subset U$, $0\le f\le 1$ and $f(y)=1$. Then
\begin{align*}
 \sup_{x\in U}m_a(x) &\ge \sup_{x\in \Gu}f(x)m_a(x)=\sup_{x\in \Gu}m_{fa}(x)=\sup_{x\in \Gu}M_{fa}(x)\\
&\ge M_{fa}(y)= M_a(y)>\sup_{x\in U}M_a(x)-\eps,
\end{align*}
finishing the proof of the lemma.
\ep

\bp[Proof of Theorem \ref{thm:singfiber}]
As we have already shown, it suffices to consider the case where $\G$ is a group bundle. If $\G$ is Hausdorff and second countable, then the result follows from~\cite{MR4043708}*{Theorem~4.12}. For the general case, let $(W_n)^\infty_{n=1}$ be a sequence of open bisections covering $\G$. For a finite subset $F\subset\mathbb N$ and numbers $\alpha_k\in\mathbb Q+i\mathbb Q$, $k\in F$, consider the function
$$
f:=\sum_{k\in F}\alpha_k\un_{W_k}\quad\text{on}\quad \G_{U(f)},\quad\text{where}\quad U(f):=\bigcap_{k\in F}s(W_k).
$$
Although $f\not\in C_c(\G_{U(f)})$ in general, we can for each $x\in U(f)$ take any function $h\in C_{c}(U(f))$ that is constantly $1$ on a neighbourhood of $x$ and set
$$
M_{f}(x):= M_{f*h}(x) \quad \text{and} \quad m_{f}(x):=m_{f*h}(x),
$$
where $M_{f*h}$ and $m_{f*h}$ are defined with respect to $\G_{U(f)}$. It is straightforward to check that~$M_{f}$ and~$m_{f}$ are independent of any choices, and since all the properties in Lemma \ref{lem:maMa} are local, they are also satisfied for $M_{f}$ and $m_{f}$. It follows that each of the sets
$$
\{x\in U(f) \mid m_f(x)\le p,\ M_f(x)\ge q\},\quad p,q\in\mathbb Q,\ p<q,
$$
is closed and has empty interior. Hence the set
$$
X (f):=\{x\in U(f) \mid m_{f}(x)=M_{f}(x) \text{ and $x$ is not a dangerous point in $U(f)$} \}
$$
is residual in $U(f)$. Let $x\in X(f)$. Then, for any $h\in C_{c}(U(f))$ that is $1$ on a neighbourhood of~$x$, we have
\begin{align}
\|\vartheta_{x,r}(f*h)\|&=\|\rho_x(f*h)\|=m_{f*h}(x)=M_{f*h}(x)\nonumber\\
&=\|\pi(f*h)_x\|=\|\vartheta_{x,e}(f*h)+J_{\sing,x}\|,\label{eq:Jsingkern}
\end{align}
where we used exactness of~\eqref{eq:exact-fiber}.

Since $X(f)$ is residual in $U(f)$, the set $U(f)\setminus X(f)$ is meager in $U(f)$, hence in $\Gu$. Take now any $x$ in the residual set $\Gu\setminus\bigcup_f(U(f)\setminus X(f))$. For each $q$ in the group algebra of $\Gxx$ over $\mathbb Q+i\mathbb Q$ there exist functions $f$ and $h\in C_{c}(U(f))$ as above with $\eta_{x}(f*h)=q$. So by density of the group algebra of $\Gxx$ over $\mathbb Q+i\mathbb Q$ in $C^*_e(\Gxx)$ and by \eqref{eq:Jsingkern} we conclude that
$$
J_{\sing,x}=\ker(C^*_e(\Gxx)\to C^*_r(\Gxx)),
$$
which finishes the proof of Theorem \ref{thm:singfiber}.
\ep

It should be said that we know very little about the isotropy fibers of $J_\sing$.

\begin{question}
Is the set $\{x: J_{\sing,x}\ne0\}$ always meager? Is it contained in the set $D$ of dangerous points?
\end{question}

We suspect the answer is no, but do  not have examples.

\bigskip

\section{Graded groupoids with essentially central isotropy}\label{sec:graded}


In this section we will use Theorem~\ref{thm:induction} to classify primitive ideals for a class of groupoids. We need some preparation to formulate the result.

\smallskip

For a locally compact \'etale groupoid $\G$ and a unit $x\in\Gu$, we will consider the closure $\overline{[x]}\subset\Gu$ of the orbit of $x$ and the open (in the relative topology) subgroupoid $\IsoGx{x}^\circ\subset\G_{\overline{[x]}}$. It should not be confused with a potentially smaller open subgroupoid $\G_{\overline{[x]}}\cap \Iso\subset \G_{\overline{[x]}}$.

The subgroup $\IsoGx{x}^\circ_x\subset\Gxx$ is always normal. We will consider groupoids such that $\IsoGx{x}^\circ_x$ is a central subgroup of $\Gxx$ for all $x\in\Gu$. Assume then that $J\subset C^*(\Gxx)$ is a primitive ideal and choose an irreducible representation $\pi$ with kernel $J$. Then the restriction of~$\pi$ to $\IsoGx{x}^\circ_x$ defines a character $\chi_J$. This character depends on $J$ but not on the choice of~$\pi$.

\smallskip

We are now ready to formulate the main result of this section.

\begin{thm}\label{thm:primitive}
Assume $\G$ is an amenable second countable Hausdorff locally compact \'etale groupoid graded by a discrete group $\Gamma$, with grading $\Phi\colon\G\to\Gamma$. Assume that for every $x\in\Gu$ the map $\Phi$ is injective on $\Gxx$ and $\Phi(\IsoGx{x}^\circ_x)\subset Z(\Gamma)$. Then all primitive ideals of $C^*(\G)$ are induced, so they correspond to the pairs $(x,J)$ consisting of $x\in\Gu$ and $J\in\Prim C^*(\Gxx)$.

Furthermore, given two such pairs $(x_1,J_1)$ and $(x_2,J_2)$, we have $\Ind J_1=\Ind J_2$ if and only if $\overline{[x_1]}=\overline{[x_2]}$ and the characters $\chi_{J_1}$ and $\chi_{J_2}$ define the same character on $\Phi(\IsoGx{x_1}^\circ_{x_1})=\Phi(\IsoGx{x_2}^\circ_{x_2})$.
\end{thm}

Note that, as we will see soon, the equality $\Phi(\IsoGx{x_1}^\circ_{x_1})=\Phi(\IsoGx{x_2}^\circ_{x_2})$ is not an extra requirement, but a consequence of the assumption $\overline{[x_1]}=\overline{[x_2]}$. We also remind that for every $J\in\Prim C^*(\Gxx)$ the induced ideal $\Ind J$ is primitive, see~\cite{IW0}.

\smallskip

When $\Gamma$ is abelian, the isotropy groups are abelian and every primitive ideal of $C^*(\Gxx)$ is given by a character of $\Gxx$. Every such character can be written as $\chi\circ\Phi$ for $\chi\in\hat\Gamma$. This explains why Theorem~\ref{thm:C} is a particular case of Theorem~\ref{thm:primitive}.

\smallskip

For the proof of the theorem we need the following, probably known, observation. Recall that a closed invariant subset of $\Gu$ is called \emph{irreducible}, if it is not the union of two proper closed invariant subsets.

\begin{lemma}\label{lem:irreducible_set}
For any locally compact \'etale groupoid $\G$ with second countable unit space $\Gu$, the unit space is irreducible if and only if $\Gu=\overline{[x]}$ for some $x$. Furthermore, then the set of such points $x$ is residual.
\end{lemma}

\bp
Assume $\Gu=\overline{[x]}$. If $\Gu=F_1\cup F_2$ for invariant closed sets $F_1$ and $F_2$, then $x\in F_i$ for some $i$ and hence $F_i=\Gu$, so $\Gu$ is irreducible. Conversely, assume $\Gu$ is irreducible. Let $(U_n)_n$ be a countable base of the topology on $\Gu$, with $U_{n}$ nonempty for all $n$. For every~$n$, the saturation $[U_n]:=r(\G_{U_n})$ of $U_n$ is an open invariant set, hence it is dense in $\Gu$ by irreducibility, as $\Gu=\overline{[U_n]}\cup (\Gu\setminus [U_n])$. Then the set $X:=\cap_n [U_n]$ is residual. The orbit of every point $x\in X$ intersects $U_n$ for all $n$, hence $[x]$ is dense in $\Gu$.
\ep

Turning to the proof of Theorem~\ref{thm:primitive}, we already know by~\cite{IW} that under our assumptions every primitive ideal is induced, so the main goal is to understand when induction gives equal ideals. In fact, we will only need to know that every primitive ideal contains an induced primitive ideal, which is considered to be the ``easy'' part of the Effros--Hahn conjecture~\cite{GR}.

So assume that $I\subset C^*(\G)$ is a primitive ideal. It is well-known that then the closed invariant set $U(I)^c$ is irreducible, see, e.g., \cite{SW}*{Proposition~2.3}. As $C^*(\G_{U(I)})\subset I$ and $C^*(\G)/C^*(\G_{U(I)})\cong C^*(\G_{U(I)^c})$, $I$ defines a primitive ideal in $C^*(\G_{U(I)^c})$ having trivial intersection with $C_0(\G^{(0)}_{U(I)^c})$. Recall also that, by Lemma~\ref{lem:orbit}, if $I=\Ind J$ for some $J\subset C^*(\Gxx)$, then $U(I)^c=\overline{[x]}$. It follows that in order to prove  Theorem~\ref{thm:primitive} it suffices to establish the following result.

\begin{prop}\label{prop:primitive}
Assume $\G$ is an amenable second countable Hausdorff locally compact \'etale groupoid graded by a discrete group $\Gamma$, with grading $\Phi\colon\G\to\Gamma$. Assume that the unit space $\Gu$ is irreducible and for every $x\in\Gu$ the map $\Phi$ is injective on $\Gxx$ and $\Phi(\Iso_x)\subset Z(\Gamma)$. Fix a point $x\in\Gu$ with dense orbit. Then the primitive ideals $I\subset C^*(\G)$ such that $C_0(\Gu)\cap I=0$ are precisely the ideals $\Ind J$ for $J\in\Prim C^*(\Gxx)$.

Furthermore, given another point $y\in\Gu$ with dense orbit, $J\in\Prim C^*(\Gxx)$ and $H\in\Prim C^*(\G^y_y)$, we have $\Ind J=\Ind H$ if and only if the characters $\chi_J$ and $\chi_H$ define the same character on $\Phi(\Iso_x)=\Phi(\Iso_y)$.
\end{prop}

For the proof let us first fix a point $x$ such that $\overline{[x]}=\Gu$ and $\Iso_x=\Gxx$, which is possible since by Lemmas~\ref{lem:iso-points} and~\ref{lem:irreducible_set} the set of such points is residual. By assumption, $A:=\Phi(\Iso_x)$ is a subgroup of $Z(\Gamma)$.

\begin{lemma}\label{lem:iso}
For any $y\in\Gu$, we have $\Phi(\Iso_y)\subset A$ and the equality holds if $y$ has dense orbit.
\end{lemma}

\bp
As $A\subset Z(\Gamma)$ and $(\Ad\gamma)(\Iso_x)=\Iso_y$ for any $\gamma\in \G^y_x$, we immediately get that $\Phi(\Iso_y)=A$ for all $y\in[x]$. Next, take any point $y\in\Gu$. If  $a\in\Phi(\Iso_y)$ for some  $a\in\Gamma$, then $a\in\Phi(\Iso_z)$ for all $z$ close to $y$, namely, for all $z$ in the open set $r(\Iso\cap\Phi^{-1}(a))$. As $[x]$ is dense in~$\Gu$, we conclude that $a\in A$. Therefore $\Phi(\Iso_y)\subset\Phi(\Iso_x)$. We haven't used that $\Iso_x=\Gxx$ in this argument, so if in addition $[y]$ is dense in $\Gu$, then by symmetry we also get $\Phi(\Iso_x)\subset \Phi(\Iso_y)$.
\ep

Assume now that $I\subset C^*_r(\G)$ is a primitive ideal such that $C_0(\Gu)\cap I=0$. By~\cite{IW}, there is a point $y\in\Gu$ and a primitive ideal $J\subset C^*(\G^y_y)$ such that $\Ind J\subset I$. By Lemma~\ref{lem:orbit}, the orbit of $y$ must be dense in $\Gu$. Consider the character $\chi_J$ of $\Iso_y$.

We need the following known result. 

\begin{lemma}\label{lem:Fell}
Assume $G$ is an amenable discrete group, $Z\subset G$ is a central subgroup, $\chi$ is a character of $Z$ and $\pi\colon G\to B(H)$ is a unitary representation of $G$ such that $\pi|_Z=\chi(\cdot)1$. Then~$\pi$ is weakly contained in $\Ind^G_Z\chi$.
\end{lemma}

\bp
Consider the quasi-regular representation $\lambda_{G/Z}\colon G\to B(\ell^2(G/Z))$. In other words, $\lambda_{G/Z}=\Ind^G_Z\eps_Z$, where $\eps_Z$ is the trivial representation of $Z$. By a version of Fell's absorption principle, see \cite{BdlHV}*{Corollary E.2.6(i)}, the representations $\lambda_{G/Z}\otimes\pi$ and $\Ind^G_Z\chi$ are quasi-equivalent. As $G/Z$ is amenable, $\eps_G$ is weakly contained in $\lambda_{G/Z}$, hence $\pi$ is weakly contained in~$\Ind^G_Z\chi$.
\ep

Returning to the proof of the proposition, we apply this lemma to $G=\G^y_y$, $Z=\Iso_y$, $\chi=\chi_J$ and $\pi$ any irreducible representation of $C^*(\G^y_y)$ with kernel $J$. Note that $\G^y_y$ is amenable by amenability of $\G$. We conclude that $\pi$ is weakly contained in $\Ind^{\G^y_y}_{\Iso_y}\chi_J$ and therefore
\begin{equation}\label{eq:ind-Gy}
\ker\Ind^\G_{\G^y_y}\Ind^{\G^y_y}_{\Iso_y}\chi_J\subset\Ind J\subset I.
\end{equation}

Let $\chi\in\hat A$ be the character such that $\chi_J=\chi\circ\Phi$ on $\Iso_y$. For every $z\in\Gu$, define
$$
\rho^\chi_z:=\Ind^\G_{\G^z_z}\Ind^{\G^z_z}_{\Iso_z}(\chi\circ\Phi).
$$
Note that this definition makes sense as $\Phi(\Iso_z)\subset A$ by Lemma~\ref{lem:iso}. Note also that both pictures of induced representations in Section~\ref{ssec:induced} make sense for subgroups of the isotropy groups, and then the usual induction-in-stages isomorphisms hold. Therefore we can equivalently define~$\rho^\chi_z$ as $\Ind^\G_{\Iso_z}(\chi\circ\Phi)$.

\begin{lemma}
The weak equivalence class of $\rho^\chi_z$ is independent of $z$ such that $\overline{[z]}=\Gu$.
\end{lemma}

\bp
It is not difficult to see that the underlying Hilbert space of $\rho^\chi_z\sim\Ind^\G_{\Iso_z}(\chi\circ\Phi)$ can be obtained from $C_c(\G_z)$ by equipping it with the pre-inner product
$$
\langle\delta_g,\delta_h\rangle_\chi:=\begin{cases}
                                       \chi(\Phi(h^{-1}g)), & \mbox{if } g\in h\Iso_z, \\
                                       0, & \mbox{otherwise},
                                     \end{cases}
$$
while the action of $C_c(\G)$ on $C_c(\G_z)$ defining $\rho^\chi_z$ is then given by the same formula~\eqref{eq:rhox} as for~$\rho_z$.

Let us show first that if $z$ is a point with dense orbit and $w_n\to z$, then $\rho^\chi_z$ is weakly contained in $\bigoplus^\infty_{n=1}\rho^\chi_{w_n}$. As $\delta_z$ is a cyclic vector for $\rho^\chi_z$, it suffices to show that
\begin{equation}\label{eq:convergence}
\langle\rho^\chi_{w_n}(f)\delta_{w_n},\delta_{w_n}\rangle_\chi\to \langle\rho^\chi_z(f)\delta_z,\delta_z\rangle_\chi\quad\text{for all}\quad f\in C_c(\G).
\end{equation}

We may assume that $f\in C_c(W)$ for an open bisection $W$. If $z\notin s(\supp f)$, then $w_n\notin s(\supp f)$ for $n$ large enough and there is nothing to prove. So assume $z\in s(\supp f)$. We then may assume that $w_n\in s(W)$ for all $n$. Let $T\colon s(W)\to W$ be the inverse of $s|_W$. We have
$$
\langle\rho^\chi_z(f)\delta_z,\delta_z\rangle_\chi=\begin{cases}
                                       f(T(z))\chi(\Phi(T(z))), & \mbox{if } T(z)\in \Iso_z, \\
                                       0, & \mbox{otherwise}.
                                     \end{cases}
$$
Therefore in order to prove~\eqref{eq:convergence} it suffices to show that $T(z)\in \Iso_z$ if and only if $T(w_n)\in \Iso_{w_n}$ for all $n$ large enough. In other words, $W\cap \Iso_z\ne\emptyset$ if and only if $W\cap  \Iso_{w_n}\ne\emptyset$ for all $n$ large enough.

The ``only if'' statement is obvious. To prove the ``if'' statement, assume that $W\cap  \Iso_z=\emptyset$. Suppose $W\cap  \Iso_{w_n}\ne\emptyset$ for arbitrary large~$n$. By passing to a subsequence we may assume that this is the case for all~$n$. Let $g_n\in\Iso_{w_n}$ be the unique elements such that $g_n\in W$. As $w_n\to z$, $z\in s(W)$ and $W$ is an open bisection, we have $g_n\to g$ for some $g\in\G^z_z\cap W$. As~$\Phi$ is continuous and $\Gamma$ is discrete, we get $\Phi(g)=\Phi(g_n)$ for sufficiently large $n$. Using that $g_n\in\Iso_{w_n}$, by Lemma~\ref{lem:iso} we conclude that $\Phi(g)\in A$. But as $z$ has a dense orbit, by the same lemma we have $\Phi(\Iso_z)=A$. Hence $g\in\Iso_z$, contradicting the assumption $W\cap\Iso_z=\emptyset$ and completing the proof of weak containment of  $\rho^\chi_z$ in $\bigoplus^\infty_{n=1}\rho^\chi_{w_n}$.

Now observe that for any $\gamma\in\G$ the map $\xi\mapsto\xi(\cdot\,\gamma)$ defines a unitary equivalence between~$\rho^\chi_{s(\gamma)}$ and~$\rho^\chi_{r(\gamma)}$, since by the centrality assumption we have $\Phi\circ\Ad\gamma=\Phi$ on $\Iso_{s(\gamma)}$. If follows that if a unit~$w$ has dense orbit and we choose $w_n\in[w]$ converging to $z$, then we can conclude that $\rho^\chi_z$ is weakly contained in $\rho^\chi_w$. Hence these representations are weakly equivalent by symmetry.
\ep

Applying this lemma to $z=x,y$ and using~\eqref{eq:ind-Gy}, we get
$$
\ker\rho^\chi_x=\ker\rho^\chi_y\subset I.
$$
But $\Iso_x=\Gxx$, so $\rho^\chi_x=\Ind^\G_{\Gxx}(\chi\circ\Phi)$ and therefore $\ker\Ind^\G_{\Gxx}(\chi\circ\Phi)\subset I$.

\begin{lemma}\label{lem:primitive-neigbour}
We have $I=\ker\Ind^\G_{\Gxx}(\chi\circ\Phi)$. Moreover, $\chi\in\hat A$ is the only character with the property $I\subset\ker\Ind^\G_{\Gxx}(\chi\circ\Phi)$.
\end{lemma}

\bp
By Theorem~\ref{thm:induction}, there is a character $\omega\in\hat A$ such that $I\subset \ker\Ind^\G_{\Gxx}(\omega\circ\Phi)$. In order to prove the lemma it suffices to show that $\omega=\chi$.

Assume $\omega\ne\chi$. Choose $a\in A$ such that $\omega(a)\ne\chi(a)$ and consider the subgroup $A':=\langle a\rangle\subset A$. Consider the open set $U:=r(\Phi^{-1}(a)\cap\Iso)\subset\Gu$ containing $x$. This set is invariant by centrality of $a\in\Gamma$. We can identify the open subgroupoid
$\Phi^{-1}(A')\cap\Iso_U\subset\G$ with $U\times A'$ using the map $g\mapsto (r(g),\Phi(g))$.

Choose functions $h\in C_c(U)_+$ and $\varphi\in C(\hat{A'})_+$ such that $h(x)=1$, $\varphi(\omega|_{A'})=1$ and $\varphi(\chi|_{A'})=0$. Consider the function $f\in C_c(U\times\hat{A'})$ defined by $f(z,\eta):=h(z)\varphi(\eta)$. We can view $f$ as a positive element of $C^*(U\times A')\cong C_0(U)\otimes C^*_r(A')\cong C_0(U\times\hat{A'})$, hence as an element of $C^*(\G_U)\subset C^*(\G)$. We then have $(\omega\circ\Phi)(\vartheta_{x,r}(f))=1$ and $(\chi\circ\Phi)(\vartheta_{z,r}(f))=0$ for all $z\in[x]$. By Lemma \ref{lem:induced-ideal2} this means that $f\notin \ker\Ind^\G_{\Gxx}(\omega\circ\Phi)\supset I$ and $f\in\ker\Ind^\G_{\Gxx}(\chi\circ\Phi)\subset I$, which is a contradiction.
\ep

\bp[Proof of Proposition~\ref{prop:primitive}]
We have shown that if $x$ is such that $\overline{[x]}=\Gu$ and $\Iso_x=\Gxx$, then for $A=\Phi(\Iso_x)$ the map $\hat A\ni\chi\mapsto \ker\Ind^\G_{\Gxx}(\chi\circ\Phi)$ into the set of primitive ideals $I$ such that $C_0(\Gu)\cap I=0$ is surjective. By Lemma~\ref{lem:primitive-neigbour} it is also injective.

Now take an arbitrary point $y\in\Gu$ with dense orbit and $J\in\Prim C^*(\G^y_y)$. Applying the above considerations to $I:=\Ind J$ we conclude that $\Ind J=\ker\Ind^\G_{\Gxx}(\chi\circ\Phi)$, where $\chi\in\hat A$ is the character such that $\chi_J=\chi\circ\Phi$ on $\Iso_y$.

To finish the proof of the proposition it remains to show that for every $\chi\in\hat A$ there is $J\in\Prim C^*(\G^y_y)$ such that $\chi_J=\chi\circ\Phi$ on $\Iso_y$. But any primitive ideal containing the kernel of the representation $\Ind^{\G^y_y}_{\Iso_y}(\chi\circ\Phi)$ of $C^*(\G^y_y)$ has this property, since the restriction of this representation to $\Iso_y$ is $(\chi\circ\Phi)(\cdot)1$.
\ep

From the proof of the proposition we see that if $y$ is a point with dense orbit, $J\in\Prim C^*(\G^y_y)$ and $\chi$ is the character on $\Phi(\Iso_y)$ such that $\chi_J=\Phi\circ\chi$ on $\Iso_y$, then
$$
\Ind J=\ker\rho^\chi_y.
$$
Therefore instead of primitive ideals of $C^*(\G^y_y)$ we can use the ideals $\ker\Ind^{\G^y_y}_{\Iso_y}(\chi\circ\Phi)$ to parameterize primitive ideals of $C^*(\G)$. This is similar to our discussion of maximal ideals in Section~\ref{sectionMaximal}, where we could use either sets of primitive ideals or $\G$-invariant ideals.

\begin{cor}\label{cor:max}
In the setting of Proposition~\ref{prop:primitive}, let $x\in\Gu$ be a point with dense orbit. Then the maximal $\G$-invariant ideals in $C^*(\Gxx)$ are precisely the ideals $\ker\Ind^{\Gxx}_{\Iso_x}\chi$ for $\chi\in\widehat{\Iso_x}$.
\end{cor}

\bp
To simplify the notation, put $G:=\Gxx$, $Z:=\Iso_x$ and $J_\chi:=\ker\Ind^G_Z\chi\subset C^*(G)$ for $\chi\in\hat Z$. We know that the map $\chi\mapsto\Ind J_\chi$ defines a bijection between $\hat Z$ and the primitive ideals $I\subset C^*(\G)$ such that $C_0(\Gu)\cap I=0$. By Lemma~\ref{lem:primitive-neigbour} we know also that the ideals $\Ind J_\chi$ are not contained in each other for different $\chi$.

Fix $\chi\in \hat Z$. By Lemma~\ref{lem:G-invariance}, the ideal $(\Ind J_\chi)_x$ is $\G$-invariant. Assume $H\subset C^*(\Gxx)$ is a proper $\G$-invariant ideal containing $(\Ind J_\chi)_x$. The ideal $H$ is contained in a primitive ideal~$K$. Then the ideal $\Ind K$ is primitive and $C_0(\Gu)\cap\Ind K=0$, hence $\Ind K=\Ind J_\omega$ for some $\omega\in\hat Z$. But then
$$
\Ind J_\chi=\Ind\big((\Ind J_\chi)_x\big)\subset\Ind H\subset\Ind K=\Ind J_\omega,
$$
which is possible only when $\omega=\chi$. Therefore $H=(\Ind H)_x\subset(\Ind K)_x=(\Ind J_\chi)_x$. This shows that $(\Ind J_\chi)_x$ is a maximal $\G$-invariant ideal.

Next we want to show that $(\Ind J_\chi)_x=J_\chi$. We know that if $H\subset C^*(G)$ is a primitive ideal, then $\Ind H=\Ind J_{\chi_H}$. From this, arguing similarly to the previous paragraph, we can conclude that $(\Ind J_\chi)_x\subset H$ if and only if $\chi=\chi_H$. It follows that
$$
(\Ind J_\chi)_x=\bigcap_{H\in\Prim C^*(G):\chi_H=\chi}H.
$$
By Lemma~\ref{lem:Fell} the last intersection contains $J_\chi$. Since $(\Ind J_\chi)_x \subset J_\chi$ by Lemma \ref{lem:induced-ideal}, we conclude that $(\Ind J_\chi)_x=J_\chi$. Therefore $J_\chi$ is a maximal $\G$-invariant ideal.


It remains to show that there are no other maximal $\G$-invariant ideals. Assume $H$ is such an ideal. Let~$K$ be a primitive ideal containing $H$. Then
$$
\Ind H\subset\Ind K= \Ind J_{\chi_K},
$$
which implies that $H\subset J_{\chi_K}$, whence $H= J_{\chi_K}$ by maximality.
\ep

\bigskip

\end{document}